\documentclass[11pt]{article}

\usepackage[T1]{fontenc}
\usepackage{amsmath,amssymb,amsthm,mathtools}
\usepackage{graphicx}
\usepackage{newtxtext}
\usepackage[subscriptcorrection]{newtxmath}
\usepackage[plain,noend]{algorithm2e}
\usepackage{subcaption}
\usepackage{natbib}
\usepackage{authblk}
\usepackage[hidelinks]{hyperref}

\makeatletter
\renewcommand{\algocf@captiontext}[2]{#1\algocf@typo. \AlCapFnt{}#2}

\def\@algocf@capt@plain{top}
\renewcommand{\algocf@makecaption}[2]{%
  \addtolength{\hsize}{\algomargin}%
  \sbox\@tempboxa{\algocf@captiontext{#1}{#2}}%
  \ifdim\wd\@tempboxa >\hsize
    \hskip .5\algomargin%
    \parbox[t]{\hsize}{\algocf@captiontext{#1}{#2}}%
  \else
    \global\@minipagefalse%
    \hbox to\hsize{\box\@tempboxa}%
  \fi
  \addtolength{\hsize}{-\algomargin}%
}
\makeatother

\newcommand{\R}{\mathbb{R}}
\newcommand{\T}{\mathrm{\scriptscriptstyle T}}
\newcommand{\Id}{\mathrm{Id}}
\newcommand{\tr}{\operatorname{tr}}

\newcommand{\Tan}{\operatorname{Tan}}
\DeclareMathOperator{\dist}{dist}

\graphicspath{{figures/}}

\theoremstyle{plain}
\newtheorem{theorem}{Theorem}
\newtheorem{lemma}{Lemma}
\newtheorem{corollary}{Corollary}
\newtheorem{proposition}{Proposition}
\newtheorem{definition}{Definition}

\theoremstyle{remark}
\newtheorem{remark}{Remark}

\title{Gradient flow for finding E-optimal designs}
\author[1]{J. Shi}
\author[2]{K.-C. Toh}
\author[3]{X. T. Tong}
\author[4]{W. K. Wong}
\affil[1]{Department of Statistics and Data Science, National University of Singapore, Singapore 117546, Singapore\\\texttt{shijieling@u.nus.edu}}
\affil[2]{Department of Mathematics, National University of Singapore, Singapore 119076, Singapore\\\texttt{mattohkc@nus.edu.sg}}
\affil[3]{Department of Mathematics, National University of Singapore, Singapore 119076, Singapore\\\texttt{xin.t.tong@nus.edu.sg}}
\affil[4]{Department of Biostatistics, Fielding School of Public Health, University of California, Los Angeles, CA 90095, U.S.A.\\\texttt{wkwong@ucla.edu}}
\date{}

\begin{document}

\maketitle

\begin{abstract}
The $E$-optimality criterion for a regression model maximizes the smallest
eigenvalue of the information matrix and becomes non-differentiable when
this eigenvalue has multiplicity greater than one. Working in the
$2$-Wasserstein space, we show that the Wasserstein gradient at an empirical
measure coincides, up to a constant factor, with the Euclidean particle
gradient for smooth criteria such as $D$- and $L$-optimality, and that the
approximation gap for equal-weight $N$-particle designs vanishes at an
explicit rate. The main challenge is the nonsmooth $E$-criterion, for which
the Wasserstein gradient does not exist. We replace it with a constrained
Wasserstein steepest-ascent field obtained by maximizing feasible directional
derivatives over the tangent cone of the design space, and prove that the
resulting flow satisfies an exact energy identity and that every limit point
is first-order stationary. The particle ascent computation reduces to a
convex semidefinite programme whose dimension equals the multiplicity of the
smallest eigenvalue. In numerical comparisons on second-order response
surface models and a seven-dimensional logistic regression model, the
constrained Wasserstein steepest-ascent method attains near-optimal
$E$-criterion values and is markedly more reliable than particle swarm
optimization in higher-dimensional settings. The framework applies more
broadly to other nonsmooth minimax criteria in optimal design, and a
numerical experiment on the minimax-single-parameter criterion confirms that
the method attains the theoretical optimum.
\end{abstract}


\section{Introduction}
Research in optimal experimental design dates back to \citet{smith1918}
and continues unabated because of rising experimental costs, model
complexity and more intricate design criteria that seek to capture the
goal of the experiment more accurately
\citep{atkinson1996usefulness,bailey,smucker2018optimal}. Given a design
criterion and a statistical model on a design space, the scientific
objective is to find a design that optimizes the criterion over all possible designs.  Consequently, 
optimal designs provide the most precise statistical inference at
minimum cost and are now used across disciplines \citep{berger}. Design
ideas are also increasingly used in emerging areas for statistical
inference, for example in the construction of optimal subsampling
strategies from massive datasets
\citep{hyjasa,johns,ortho,jesus,jmlr}.

Following convention, the worth of a design is measured by the information matrix constructed from the design and the regression model. Approximate design theory, pioneered by Jack Kiefer and documented extensively in his collected works \citep{kiefer85}, treats a design as a probability measure on a compact design space, and shows that when the criterion is a convex or concave functional on the space of information matrices, there is a unified theory for finding many types of optimal approximate designs in linear models, together with algorithms whose convergence to an optimum can be established. Design monographs, such as \citet{fedorov1972}, \citet{pazman}, \citet{pukel2008optimal}, and \citet{berger2009introduction} contain technical details and illustrations. When the model is nonlinear or the criterion is nonsmooth, as with minimax-type criteria, the analytical problem is substantially harder and computational difficulties arise. This paper develops a unified framework, based on the geometry of the $2$-Wasserstein space, for computing optimal designs under both smooth and nonsmooth criteria, with particular attention to the $E$-optimality criterion and more general nonsmooth minimax criteria, for which a convergence guarantee to a stationary design is established.

Let $\Omega\subset\R^d$ be a compact design space and let
$f:\Omega\to\R^m$ be continuously differentiable. For a linear model
with mean response $E(y\mid x)=f(x)^\T\theta$, the vector $f$ is the
regression vector. For a nonlinear model with mean response
$\mu(x,\theta)$, we take $f(x)=\partial\mu(x,\theta)/\partial\theta$
evaluated at a nominal parameter value $\theta^\ast$, in which case the
designs constructed below are locally optimal at $\theta^\ast$. For a
design $\rho\in\mathcal{P}_2(\Omega)$, the set of probability measures
on $\Omega$ with finite second moment, the information matrix is
proportional to
\[
M_{\rho}=\int_{\Omega}f(x)f(x)^\T\,\rho(dx).
\]
The $E$-optimal design problem is
\begin{equation} \label{eq:eoptimal}
\max_{\rho \in \mathcal{P}_2(\Omega)} \mathcal{F}_E(\rho), \qquad \mathcal{F}_E(\rho)=\lambda_{\min}(M_{\rho}),
\end{equation}
where $\lambda_{\min}$ denotes the smallest eigenvalue. We also consider the smooth criteria
\begin{equation} \label{eq:othercriteria}
\mathcal{F}_L(\rho)=\tr(LM_{\rho}^{-1}), \qquad \mathcal{F}_D(\rho)=\log \det(M_{\rho}),
\end{equation}
with $L$ fixed and positive semidefinite. Here $D$-optimality corresponds to maximizing $\mathcal{F}_D$, whereas $L$-optimality corresponds to minimizing $\mathcal{F}_L$.

 Among these, $D$-optimality is the most widely studied: the criterion is smooth and the resulting designs admit relatively tractable analytical characterizations. Analytical derivations, however, tend to be model-specific and do not extend readily even to minor perturbations of the model; see \citet{dette1994optimal} for linear and nonlinear two-parameter models and \citet{mixcubnofull} and \citet{mixcubfull} for Scheff\'{e} mixture polynomials. The $E$-criterion is substantially harder because $\lambda_{\min}(M_\rho)$ is non-differentiable whenever the smallest eigenvalue has multiplicity greater than one \citep{eexponential,estandardized,dette_e-optimal_2014}. Under normality, an $E$-optimal design minimizes the length of the longest principal axis of the confidence ellipsoid for the model parameters, and in certain settings can yield more efficient parameter estimation than $D$-optimal designs \citep{eexponential}.

 Early analytical work on $E$-optimality focused on linear models. \citet{heiligers} and \citet{melas} derived $E$-optimal designs for polynomial regression over various intervals, \citet{jpmorgan} constructed $E$-optimal designs for three-group comparisons, and \citet{dette_e-optimal_2014} obtained $E$-optimal designs for second-order polynomial models with several factors. For nonlinear models, analytical results are available only in simple cases: \citet{dettewong} found $E$-optimal designs for the two-parameter Michaelis--Menten model, and \citet{estandardized} extended the maximin criterion to find standardized maximin $E$-optimal designs for the same model. Our interest in $E$-optimality is further motivated by its recent diverse applications, including clinical trials \citep{biophar,harman}, food science \citep{berkholz2000data,smets2002optimal,maths}, genetic studies \citep{kerr2012optimality}, robustified biological experiments \citep{jordan}, and sensor placement with non-supermodular objectives \citep{super}. 
 There is also mathematical interest in $E$-optimal designs for
structured linear models and in efficient algorithms for the
$E$-criterion: \citet{linalgebra} constructed $E$-optimal designs for
chemical balance weighing and two- and three-level factorial models
under correlated errors with known correlation structure,
\citet{geometry} gave a geometrical construction of $E$-optimal designs
in low dimensions, and \citet{monotonic} developed a monotonic
majorization--minimization algorithm with closed-form updates for the
$E$-optimal design problem. The $E$-criterion is further connected to
designs minimizing the condition number of the information matrix
\citep{condition2,condition1,yue2023constructing}.

Traditional algorithms for computing optimal designs include Fedorov--Wynn exchange algorithms \citep{wynn1970sequential,wynn1972results} and their variants \citep{dette_improving_2008,fedbrazil,mandal2015algorithmic}; see the monographs of \citet{fedorov1972}, \citet{pazman}, \citet{atkinson2007optimum} and \citet{berger2009introduction}. These methods may perform poorly when the criterion is nonsmooth, the model is nonlinear, or the dimension of the optimization problem is high \citep{larntz,royal}. \citet{tonyng} proposed algorithms with proofs of convergence for finding $A$- and $D$-optimal designs, but like all earlier proofs, they restricted attention to differentiable criteria and for linear models only. More broadly, nonsmooth minimax criteria arise in many design settings, including minimax $D$-optimal designs for nonlinear models \citep{king_minimax_2000,berger_minimax_2000}, minimax predicted variance designs under heteroscedasticity \citep{brown_algorithmic_2000,wong_multifactor_1994}, and minimax-single-parameter designs \citep{murty_minimax_1971,wong_unified_1992}. 

Nature-inspired metaheuristics, such as particle swarm optimization and
differential evolution, offer derivative-free alternatives and are
widely used across a broad range of optimization problems
\citep{whitacre2011recent,whitacre2011survival}. Recent applications span diverse fields, including the prediction of progressive lung regions at 6- to 12-month follow-up from single-time-point computed tomography scans in patients with idiopathic pulmonary fibrosis \citep{shiart}, the control and tracking of COVID-19 spread \citep{covid}, and the improvement of statistical estimation procedures \citep{parkshaowong}. In the design literature, they have been used to solve a variety of problems. \citet{wire} and \citet{zack1} provide tutorial reviews on the use of PSO and differential evolution for challenging design problems. Specific applications
include finding efficient computer experiments \citep{santer},
high-dimensional locally $D$-optimal logistic designs \citep{stella},
$G$-optimal designs for hierarchical linear models \citep{soft},
order-of-addition designs, which are drug-combination studies in which
the order of drug administration is itself a design variable
\citep{zack2}, and more practical early-phase clinical trial designs,
where optimal designs are found subject to realistic constraints
\citep{ct} or are multi-stage with control of both efficacy and
toxicity rates \citep{swarmjcgs}. However, metaheuristics lack
convergence guarantees and require proper tuning for solving more
complex optimization problems.

In this paper, we develop a framework based on the geometry of the $2$-Wasserstein space for optimal experimental design. For smooth criteria such as $\mathcal{F}_D$ and $\mathcal{F}_L$, the Wasserstein gradient at an empirical measure coincides, up to a constant factor, with the Euclidean particle gradient, and the approximation gap for equal-weight $N$-particle designs vanishes at an explicit rate (Section~\ref{sec:smooth}). Our focus is on the nonsmooth $E$-optimality design criterion, but the method applies to similar criteria. The main contributions are as follows.

\smallskip
\noindent
(i) When $\lambda_{\min}(M_\rho)$ has multiplicity greater than one, the
Wasserstein gradient of $\mathcal{F}_E$ does not exist. We introduce a
constrained Wasserstein steepest-ascent field that replaces it, constructed
by maximizing a feasible directional derivative over the tangent cone of
$\Omega$ (Section~\ref{sec:nonsmooth}). The variational representation
$\lambda_{\min}(M_\rho) = \min_{G \in \Delta} \tr(G M_\rho)$
endows $\mathcal{F}_E$ with a min-type structure. Our construction extends
the classical finite-dimensional steepest-descent framework for such
objectives \citep[Ch.~V, \S~8]{demyanov_introduction_1974} to the $2$-Wasserstein space,
where feasibility is enforced by projection onto the tangent cone of the
design space.

\smallskip
\noindent
(ii) We prove that the flow driven by this field satisfies an exact energy
identity and that every $W_2$-limit point of the trajectory is stationary.
To our knowledge, this is the first convergence result for a
Wasserstein-based method applied to $E$-optimal design.

\smallskip
\noindent
(iii) By exploiting the variational characterization of $\lambda_{\min}$ and
the structure of its minimizing eigenspace, we reduce the particle
steepest-ascent computation to a convex semidefinite programme whose
dimension equals the multiplicity of the smallest eigenvalue. Numerical
experiments (Section~\ref{sec:experiments}) on second-order response surface
models and a seven-dimensional logistic regression demonstrate that the
method attains near-optimal $E$-designs and is markedly more reliable than
particle swarm optimization in higher-dimensional settings.

\smallskip
\noindent
(iv) The constrained Wasserstein steepest-ascent framework applies more
broadly to other nonsmooth minimax criteria in optimal design. We
illustrate this with the minimax-single-parameter criterion
(Section~\ref{sec:scaled-param}), for which the direction subproblem
reduces to a convex programme over a probability simplex. A numerical
experiment (Section~\ref{sec:sp-experiment}) confirms that the method
attains the theoretical optimum for this criterion.

\medskip
\textit{Notation.}
The symbol $\|\cdot\|$ denotes the Euclidean norm on $\R^d$. For a
probability measure $\rho$ on $\Omega$, let $L^2(\rho;\R^d)$ denote the
Hilbert space of Borel vector fields $u:\Omega\to\R^d$, equipped with
\[
\|u\|_{\rho}^2=\int_{\Omega}\|u(x)\|^2\,\rho(dx),
\qquad
\langle u,v\rangle_{\rho}=\int_{\Omega}u(x)^\T v(x)\,\rho(dx).
\]
The symbol $C^1(\Omega;\R^m)$ denotes the continuously differentiable
functions from $\Omega$ to $\R^m$, $C_c^\infty(\R^d)$ the smooth
compactly supported scalar functions on $\R^d$, and
$\nabla f(x)\in\R^{m\times d}$ the Jacobian of $f$ with respect to $x$.
We write $\mathbb{S}^m$ for the space of $m\times m$ real symmetric
matrices and $\mathbb{S}_+^m$ for the cone of positive semidefinite
matrices in $\mathbb{S}^m$.

\section{Wasserstein gradients and particle approximation}
\label{sec:smooth}

\subsection{Wasserstein gradients for design criteria}
\label{sec:smooth-wassdiff}
The design criteria in \eqref{eq:eoptimal} and \eqref{eq:othercriteria}
are functionals on $\mathcal{P}_2(\Omega)$, the space of probability
measures on the design region with finite second moment. Equipping this
space with the $2$-Wasserstein metric provides a notion of Wasserstein
gradient that leads naturally to particle-based optimization. This
subsection recalls the relevant definitions and computes the Wasserstein
gradients of these criteria; subsequent subsections establish a
particle-gradient equivalence (Section~\ref{sec:smooth-flow}) and an
approximation rate for equal-weight particle designs
(Section~\ref{sec:particle-approx}).

Let $\mathcal{P}(\R^d)$ denote the Borel probability measures on $\R^d$, and let
\[
\mathcal{P}_2(\R^d)
:=
\left\{
\rho \in \mathcal{P}(\R^d):
\int_{\R^d}\|x\|^2\,\rho(dx)<\infty
\right\}.
\]
The subset $\mathcal{P}_2(\Omega)$ is defined analogously. Since $\Omega$ is
compact, $\mathcal{P}_2(\Omega)$ coincides with the set of all Borel
probability measures on $\Omega$.
For a Borel map $T:\R^d\to\R^d$, the pushforward $T_{\#}\rho$ is defined by $(T_{\#}\rho)(A)=\rho(T^{-1}(A))$ for Borel sets $A$. If $\Gamma(\rho,\rho')$ denotes the set of couplings of $\rho$ and $\rho'$, the $2$-Wasserstein distance is
\begin{equation}
\label{eq:wass_def_eq}
W_2(\rho,\rho')
:=
\left(
\min_{\gamma \in \Gamma(\rho,\rho')}
\int_{\R^d\times\R^d}\|x-y\|^2\,\gamma(dx,dy)
\right)^{1/2};
\end{equation}
see \citet[Section~7.1]{villani2009optimal}. The tangent space at $\rho \in \mathcal{P}_2(\R^d)$ is
\[
\Tan_{\rho}\mathcal{P}_2(\R^d)
:=
\overline{\{\nabla \varphi:\varphi\in C_c^\infty(\R^d)\}}^{\,L^2(\rho;\R^d)}.
\]

\begin{definition}[Wasserstein differentiability]
\label{def:Wassersteingrad}
Let $\mathcal{F}:\mathcal{P}_2(\R^d)\to\R$. A vector field $\nabla_{W_2}\mathcal{F}(\rho)\in \Tan_{\rho}\mathcal{P}_2(\R^d)$ is called the Wasserstein gradient of $\mathcal{F}$ at $\rho$ if, for every $\rho' \in \mathcal{P}_2(\R^d)$ and every coupling $\gamma \in \Gamma(\rho,\rho')$ attaining the minimum in \eqref{eq:wass_def_eq},
\[
\mathcal{F}(\rho')-\mathcal{F}(\rho)
=
\int_{\R^d\times\R^d}
\left\langle \nabla_{W_2}\mathcal{F}(\rho)(x),\,y-x\right\rangle
\gamma(dx,dy)
+o\!\left(W_2(\rho,\rho')\right)
\]
as $\rho' \to \rho$ in $W_2$.
\end{definition}

For the smooth criteria in \eqref{eq:othercriteria}, direct differentiation gives
\begin{align*}
\nabla_{W_2}\mathcal{F}_L(\rho)(x)
&=
-2\nabla f(x)^\T M_{\rho}^{-1}LM_{\rho}^{-1}f(x),
\\
\nabla_{W_2}\mathcal{F}_D(\rho)(x)
&=
2\nabla f(x)^\T M_{\rho}^{-1}f(x).
\end{align*}
If $M_{\rho}$ has a simple smallest eigenvalue with unit eigenvector $q_{\rho}$, then the $E$-criterion is differentiable at $\rho$ and
\begin{equation}
\label{eq:E-simple-gradient}
\nabla_{W_2}\mathcal{F}_E(\rho)(x)
=
2\bigl(q_{\rho}^\T f(x)\bigr)\nabla f(x)^\T q_{\rho}.
\end{equation}
When $\lambda_{\min}(M_\rho)$ has multiplicity greater than one, $\mathcal{F}_E$ is not Wasserstein differentiable at $\rho$ and the gradient formulas above do not apply. A different framework, based on directional derivatives and the tangent cone of $\Omega$, is developed in Section~\ref{sec:nonsmooth}.

\subsection{Particle Wasserstein gradient flow}
\label{sec:smooth-flow}
Let $x=(x_1,\ldots,x_N)\in(\R^d)^N$ and let
\[
\rho_N
:=
\frac{1}{N}\sum_{i=1}^N \delta_{x_i},
\qquad
\mathcal{F}_N(x):=\mathcal{F}(\rho_N).
\]
For empirical measures, the Wasserstein gradient and the Euclidean gradient of $\mathcal{F}_N$ agree at the particle locations up to the factor $1/N$. This is the structural reason why particle methods provide a faithful discretization of the Wasserstein flow.

\begin{proposition}[Particle-gradient equivalence]
\label{prop:particle_gradient}
Assume that the particle locations $x_1,\ldots,x_N$ are pairwise distinct. Let $\rho_N=N^{-1}\sum_{i=1}^N\delta_{x_i}$ and suppose that $\mathcal{F}$ is Wasserstein differentiable at $\rho_N$ and that $\mathcal{F}_N$ is differentiable at $x$. Then
\begin{equation}
\label{eq:particle-gradient}
\nabla_{x_i}\mathcal{F}_N(x)
=
\frac{1}{N}\nabla_{W_2}\mathcal{F}(\rho_N)(x_i),
\qquad
i=1,\ldots,N.
\end{equation}
\end{proposition}

Recall that, in the standard sense of \citet[Section~11.1.1]{ambrosio_gradient_2008}, a Wasserstein gradient flow of a functional $\mathcal{F}$ is a curve $(\rho_t)_{t\ge 0}$ satisfying
\[
\partial_t\rho_t + \nabla\cdot(\rho_t v_t) = 0,
\qquad
v_t = -\nabla_{W_2}\mathcal{F}(\rho_t),
\]
corresponding to minimization of $\mathcal{F}$. The following corollary is stated under this convention; the maximization case follows by applying the same result to $-\mathcal{F}$.

\begin{corollary}[Empirical flows]
\label{cor:particle_lift}
Under the assumptions of Proposition~\ref{prop:particle_gradient}, suppose that $(\rho_t)_{t\ge 0}$ is a Wasserstein gradient flow of $\mathcal{F}$ with $\rho_0=\rho_N$. Then $\rho_t$ remains empirical:
\[
\rho_t=\frac{1}{N}\sum_{i=1}^N\delta_{x_i(t)},
\]
and the particle positions satisfy the time-rescaled Euclidean gradient flow
\[
\frac{d}{dt}x_i(t)
=
-N\nabla_{x_i}\mathcal{F}_N(x(t)),
\qquad
i=1,\ldots,N.
\]
\end{corollary}

\begin{proof}
The claim follows by substituting \eqref{eq:particle-gradient} into the continuity equation.
\end{proof}

\subsection{Optimal-value approximation by equal-weight particle designs}
\label{sec:particle-approx}
For $N\ge 1$, let
\[
\mathcal{P}_N(\Omega)
:=
\left\{
\frac{1}{N}\sum_{i=1}^N \delta_{x_i}: x_i\in\Omega
\right\}.
\]
The next proposition shows that, under a local $W_2$-Lipschitz condition at an optimizer, restricting the design search to equal-weight $N$-particle measures incurs only an explicit gap in the optimal value. This provides a basic justification for the particle approximation used in the paper. Although the statement is written for minimization, the corresponding maximization bound follows immediately by applying it to $-\mathcal{F}$.

\begin{proposition}[Optimal-value approximation]
\label{prop:particle_value}
Assume that $\Omega$ is compact. Let $\mathcal{F}:\mathcal{P}_2(\Omega)\to\R$ and let $\rho^\ast$ minimizes $\mathcal{F}$ over $\mathcal{P}_2(\Omega)$. Suppose that $\mathcal{F}$ is locally $W_2$-Lipschitz at $\rho^\ast$,
in the sense that there exist $L,\delta>0$ such that
$|\mathcal{F}(\rho)-\mathcal{F}(\rho^\ast)|
\le L\,W_2(\rho,\rho^\ast)$
whenever $W_2(\rho,\rho^\ast)\le \delta$.
Then there exists $C>0$ such that, for all sufficiently large $N$,
\begin{equation}
\label{eq:particle-value-rate}
0
\le
\inf_{\rho_N\in \mathcal{P}_N(\Omega)} \mathcal{F}(\rho_N)
-\inf_{\rho\in \mathcal{P}_2(\Omega)} \mathcal{F}(\rho)
\le
C\,r_N,
\end{equation}
where
\[
r_N
:=
\begin{cases}
N^{-1/4}, & d<4,\\[0.3ex]
N^{-1/4}(\log N)^{1/2}, & d=4,\\[0.3ex]
N^{-1/d}, & d>4.
\end{cases}
\]
\end{proposition}

A proof is given in the appendix.

\begin{remark}
\label{rem:section2-transition}
The differentiability-based results in
Sections~\ref{sec:smooth-wassdiff}--\ref{sec:smooth-flow} apply to
Wasserstein-differentiable criteria, including $\mathcal{F}_L$ and
$\mathcal{F}_D$, and to $\mathcal{F}_E$ only when $\lambda_{\min}(M_\rho)$
is simple. Proposition~\ref{prop:particle_value}, by contrast, requires only
local $W_2$-Lipschitz continuity at an optimizer. When the smallest eigenvalue
has multiplicity greater than one, the Wasserstein gradient does not exist and
a different approach is needed; this is developed in
Section~\ref{sec:nonsmooth}.
\end{remark}

\section[Constrained Wasserstein steepest-ascent flow for the E-optimal design problem]{Constrained Wasserstein steepest-ascent flow for the $E$-optimal design problem}
\label{sec:nonsmooth}

\subsection[Directional derivative of the E-criterion]{Directional derivative of the $E$-criterion}
\label{sec:directional-derivative}
The Wasserstein gradient flow of Section~\ref{sec:smooth} relies on the Wasserstein
differentiability of the design criterion. When $\lambda_{\min}(M_\rho)$
has multiplicity greater than one, the $E$-criterion $\mathcal{F}_E$ is no
longer Wasserstein differentiable and a different approach is needed.
Following the classical steepest-descent framework for nonsmooth min-type
objectives in finite dimensions \citep[Ch.~V, \S~8]{demyanov_introduction_1974}, we
define a directional derivative $D\mathcal{F}_E(\rho)[u]$ for feasible
velocity fields $u$, select the direction of steepest ascent by maximizing
over all feasible directions of unit norm, and use the resulting field to
drive a continuous-time flow. This subsection derives the directional derivative; subsequent subsections
construct the constrained Wasserstein steepest-ascent field and define
stationarity (Section~\ref{sec:steepest-ascent}), establish an energy
identity and limit-point stationarity for the associated flow
(Section~\ref{sec:flow}), and reduce the particle-level direction
computation to a semidefinite programme
(Section~\ref{sec:particle-computation}).

Throughout this section, $\Omega\subset\R^d$ is nonempty, compact, and
convex. The goal is to obtain an explicit formula for the time derivative of
$\mathcal{F}_E$ along feasible flows, which will define the directional
derivative $D\mathcal{F}_E(\rho)[u]$.

For $x\in\Omega$, let
\[
T_{\Omega}(x):=\overline{\{\alpha(y-x):\alpha\ge 0,\ y\in\Omega\}}
\]
denote the tangent cone of $\Omega$ at $x$, the closure of all directions
pointing from $x$ into $\Omega$. A standard variational characterization
\citep[Section~4.2]{HornJohnson2012} gives, for every symmetric matrix $M$,
\begin{equation}
\label{eq:lambda-min-trace}
\lambda_{\min}(M)=\min_{G\in\Delta}\tr(GM),
\qquad
\Delta:=\{G\in\mathbb{S}_+^m:\tr(G)=1\}.
\end{equation}
We write
\[
\mathcal{G}(M):=\arg\min_{G\in\Delta}\tr(GM)
\]
for the set of minimizers in \eqref{eq:lambda-min-trace}. The next
proposition is a direct consequence of \eqref{eq:lambda-min-trace} and the
spectral theorem; see \citet[Section~4.1]{HornJohnson2012}.

\begin{proposition}[Structure of the minimizer set]
\label{prop:G-structure}
Let $M\in\mathbb{S}^m$, and let $V\in\R^{m\times s}$ have orthonormal
columns spanning the eigenspace of $M$ associated with $\lambda_{\min}(M)$.
Then
\[
\mathcal{G}(M)=\{VSV^\T:S\in\mathbb{S}_+^{s},\ \tr(S)=1\}.
\]
In particular, if $\lambda_{\min}(M)$ is simple with unit eigenvector $q$,
then $\mathcal{G}(M)=\{qq^\T\}$.
\end{proposition}

For $G\in\Delta$, define the vector field
\[
a_G(x):=2\nabla f(x)^\T G f(x),
\qquad x\in\Omega.
\]
For $\rho\in\mathcal{P}_2(\Omega)$ and any velocity field
$u\in L^2(\rho;\R^d)$ satisfying $u(x)\in T_{\Omega}(x)$ for $\rho$-almost
every $x$, so that the induced motion keeps the support within $\Omega$,
define
\begin{equation}
\label{eq:directional-derivative}
D\mathcal{F}_E(\rho)[u]
:=
\min_{G\in \mathcal{G}(M_{\rho})}\langle a_G,u\rangle_{\rho}.
\end{equation}
\begin{proposition}[Chain rule]
\label{prop:directional-derivative}
Let $(\rho_t)_{|t|<\varepsilon}\subset\mathcal{P}_2(\Omega)$ be an
absolutely continuous curve, and let $u_t\in L^2(\rho_t;\R^d)$ be a
velocity field with $u_t(x)\in T_\Omega(x)$ for $\rho_t$-almost every
$x$ and almost every $t$, satisfying
$\partial_t\rho_t+\nabla\!\cdot(\rho_tu_t)=0$ distributionally on
$(-\varepsilon,\varepsilon)\times\Omega$. Then
$t\mapsto\mathcal{F}_E(\rho_t)$ is absolutely continuous and, for
almost every $t\in(-\varepsilon,\varepsilon)$,
\[
\frac{d}{dt}\mathcal{F}_E(\rho_t)
=
D\mathcal{F}_E(\rho_t)[u_t].
\]
\end{proposition}

In particular, the time derivative of $\mathcal{F}_E$ along any feasible
flow is determined for almost every $t$ solely by the pair
$(\rho_t, u_t)$ through $D\mathcal{F}_E(\rho_t)[u_t]$. We therefore refer
to $D\mathcal{F}_E(\rho)[u]$ as the directional derivative of
$\mathcal{F}_E$ at $\rho$ along the feasible direction $u$. A proof is
given in the appendix.

\begin{remark}[Simple-eigenvalue case]
\label{rem:simple-eigenvalue-directional}
If $M_{\rho}$ has a simple smallest eigenvalue with unit eigenvector
$q_{\rho}$, then $\mathcal{G}(M_{\rho})=\{q_{\rho}q_{\rho}^\T\}$ and
\[
D\mathcal{F}_E(\rho)[u]
=
\left\langle 2\nabla f(\cdot)^\T q_{\rho}q_{\rho}^\T
f(\cdot),u\right\rangle_{\rho}.
\]
Hence the directional derivative reduces to pairing with the smooth
Wasserstein gradient in \eqref{eq:E-simple-gradient}.
\end{remark}

\subsection{Constrained Wasserstein steepest-ascent field and stationarity}
\label{sec:steepest-ascent}
With the directional derivative in hand, we now construct a constrained
Wasserstein steepest-ascent field to replace the unavailable Wasserstein
gradient. The idea is to select, among all feasible directions, the one
that increases $\mathcal{F}_E$ most rapidly. Proofs of the results stated
in this subsection are given in the appendix.

To formalize this, let the feasible unit ball at
$\rho\in\mathcal{P}_2(\Omega)$ be
\[
K_{\rho}
:=
\left\{
u\in L^2(\rho;\R^d):
\|u\|_{\rho}\le 1,\;
u(x)\in T_{\Omega}(x)
\right\},
\]
where the constraint $u(x)\in T_{\Omega}(x)$ is required for
$\rho$-almost every $x$. Define
\[
m_{\Omega}(\rho)
:=
\sup_{u\in K_{\rho}}D\mathcal{F}_E(\rho)[u],
\]
the maximum rate at which $\mathcal{F}_E$ can increase over feasible
directions at $\rho$.

\begin{definition}[Constrained Wasserstein steepest-ascent direction and
field]
\label{def:steepest-ascent}
A constrained Wasserstein steepest-ascent direction at $\rho$ is any
maximizer
\[
u_{\rho}^{\star}
\in
\arg\max_{u\in K_{\rho}}D\mathcal{F}_E(\rho)[u].
\]
The associated constrained Wasserstein steepest-ascent field is
\[
\bar{\nabla}_{W_2}\mathcal{F}_E(\rho)
:=
m_{\Omega}(\rho)\,u_{\rho}^{\star}.
\]
\end{definition}

Computing $u_\rho^\star$ directly from
Definition~\ref{def:steepest-ascent} requires maximizing over the
infinite-dimensional ball $K_\rho$, where the objective
$D\mathcal{F}_E(\rho)[u] = \min_{G\in\mathcal{G}(M_\rho)}
\langle a_G, u\rangle_\rho$ is itself a minimum over
$\mathcal{G}(M_\rho)$. The next proposition shows that the steepest-ascent
field can be obtained by first solving a finite-dimensional minimization
over $\mathcal{G}(M_\rho)$ and then projecting onto the tangent cone.

\begin{proposition}[Construction of the steepest-ascent field]
\label{prop:gap-representation}
For every $\rho\in\mathcal{P}_2(\Omega)$, there exists a minimizer
\begin{equation}
\label{eq:Gstar-def}
G_{\rho}^{\star}
\in
\arg\min_{G\in \mathcal{G}(M_{\rho})}
\left\|
\Pi_{T_{\Omega}(\cdot)}\bigl(a_G(\cdot)\bigr)
\right\|_{\rho},
\end{equation}
and the constrained Wasserstein steepest-ascent field is \begin{equation}
\label{eq:steepest-field-explicit}
\bar{\nabla}_{W_2}\mathcal{F}_E(\rho)
=
\Pi_{T_{\Omega}(\cdot)}\bigl(a_{G_{\rho}^{\star}}(\cdot)\bigr).
\end{equation}
\end{proposition}

The representation \eqref{eq:steepest-field-explicit} shows that the
steepest-ascent field is the projection of $a_{G_\rho^\star}$ onto the
tangent cone of $\Omega$, exactly as a projected gradient.

\begin{remark}[Simple-eigenvalue case]
\label{rem:simple-eigenvalue}
If $M_{\rho}$ has a simple smallest eigenvalue with unit eigenvector
$q_{\rho}$, then $\mathcal{G}(M_\rho)=\{q_\rho q_\rho^\T\}$ and
\[
\bar{\nabla}_{W_2}\mathcal{F}_E(\rho)
=
\Pi_{T_{\Omega}(\cdot)}
\bigl(
2\nabla f(\cdot)^\T q_{\rho}q_{\rho}^\T f(\cdot)
\bigr),
\]
so the constrained Wasserstein steepest-ascent field reduces to the
projection of the smooth Wasserstein gradient onto the tangent cone of
$\Omega$.
\end{remark}

The Wasserstein steepest-ascent field constructed above will drive the continuous-time
flow studied in Section~\ref{sec:flow}. To state the convergence guarantee for that
flow, we need a first-order optimality condition for the constrained
$E$-optimal design problem: a measure $\rho$ should be declared stationary
when $\bar{\nabla}_{W_2}\mathcal{F}_E(\rho)=0$, i.e., when no feasible
direction can increase $\mathcal{F}_E$.

\begin{definition}[Stationary point]
\label{def:stationary}
A measure $\rho\in\mathcal{P}_2(\Omega)$ is stationary for the constrained
$E$-optimal design problem if
\[
D\mathcal{F}_E(\rho)[w]\le 0
\]
for every feasible direction $w\in L^2(\rho;\R^d)$ such that
$w(x)\in T_{\Omega}(x)$ for $\rho$-almost every $x$.
\end{definition}

\begin{proposition}[Stationarity characterized by the steepest-ascent field]
\label{prop:gap-stationary}
For every $\rho\in\mathcal{P}_2(\Omega)$,
$\bar{\nabla}_{W_2}\mathcal{F}_E(\rho)=0$ if and only if $\rho$ is
stationary in the sense of Definition~\ref{def:stationary}.
\end{proposition}

\begin{remark}
\label{rem:stationarity-measure}
Since $\bar{\nabla}_{W_2}\mathcal{F}_E(\rho) = m_\Omega(\rho)\,u_\rho^\star$,
the condition $\bar{\nabla}_{W_2}\mathcal{F}_E(\rho)=0$ is equivalent to
$m_\Omega(\rho)=0$. We refer to $m_\Omega(\rho)$ as the stationarity
measure. Being a nonnegative scalar, it is well suited both as the quantity
controlling the convergence analysis of the flow in Section~\ref{sec:flow} and as a
stopping criterion in the particle algorithm  of Section~\ref{sec:particle-computation}.
\end{remark}

\subsection{Constrained Wasserstein steepest-ascent flow and limit-point
stationarity}
\label{sec:flow}

Section~\ref{sec:steepest-ascent} constructed the constrained Wasserstein
steepest-ascent field $\bar{\nabla}_{W_2}\mathcal{F}_E(\rho)$. We now use it
to drive a continuous-time flow and establish its asymptotic first-order
properties.

The constrained Wasserstein steepest-ascent flow is defined by
\begin{equation}
\label{eq:flow_bar}
\partial_t \rho_t+\nabla\!\cdot(\rho_t v_t)=0,
\qquad
v_t=\bar{\nabla}_{W_2}\mathcal{F}_E(\rho_t).
\end{equation}
If $(\rho_t)$ is absolutely continuous and satisfies \eqref{eq:flow_bar},
then Proposition~\ref{prop:directional-derivative} yields that
$t\mapsto \mathcal{F}_E(\rho_t)$ is absolutely continuous and
\[
\frac{d}{dt}\mathcal{F}_E(\rho_t)
=
D\mathcal{F}_E(\rho_t)[v_t]
\]
for almost every $t$. Since $v_t$ is chosen as the constrained Wasserstein
steepest-ascent field, this identity leads directly to the energy law below.

\begin{proposition}[Energy identity]
\label{prop:energy-identity-nonsmooth}
Let $(\rho_t)_{t\in[0,T]}\subset\mathcal{P}_2(\Omega)$ be an absolutely
continuous curve satisfying \eqref{eq:flow_bar}. Then, for almost every
$t\in[0,T]$,
\[
\frac{d}{dt}\mathcal{F}_E(\rho_t)
=
\|v_t\|_{\rho_t}^2.
\]
\end{proposition}

\begin{proof}
Since $v_t = \bar{\nabla}_{W_2}\mathcal{F}_E(\rho_t)$ by
\eqref{eq:flow_bar}, Proposition~\ref{prop:directional-derivative} gives
$\frac{d}{dt}\mathcal{F}_E(\rho_t) = D\mathcal{F}_E(\rho_t)[v_t]$ for
almost every $t\in[0,T]$. Fix any such $t$ and let $u^\star_{\rho_t}$ be as
in Definition~\ref{def:steepest-ascent}, so that
$v_t = m_{\Omega}(\rho_t)\,u^\star_{\rho_t}$. The maximizing property of
$u^\star_{\rho_t}$ gives
$D\mathcal{F}_E(\rho_t)[u^\star_{\rho_t}] = m_{\Omega}(\rho_t)$, and
positive homogeneity of \eqref{eq:directional-derivative} in the direction
variable yields
\[
\frac{d}{dt}\mathcal{F}_E(\rho_t)
=
D\mathcal{F}_E(\rho_t)[v_t]
=
m_{\Omega}(\rho_t)\,D\mathcal{F}_E(\rho_t)[u^\star_{\rho_t}]
=
m_{\Omega}(\rho_t)^2.
\]
It remains to show that $\|v_t\|_{\rho_t}^2 = m_\Omega(\rho_t)^2$. If
$m_\Omega(\rho_t) = 0$, then $v_t = 0$ and both sides vanish. If
$m_\Omega(\rho_t) > 0$, then $\|u^\star_{\rho_t}\|_{\rho_t} = 1$, since
otherwise $u^\star_{\rho_t}/\|u^\star_{\rho_t}\|_{\rho_t} \in K_{\rho_t}$
would yield a strictly larger value of $D\mathcal{F}_E(\rho_t)[u]$ by
positive homogeneity, contradicting maximality. Therefore
$\|v_t\|_{\rho_t}^2 = m_\Omega(\rho_t)^2\|u^\star_{\rho_t}\|_{\rho_t}^2
= m_\Omega(\rho_t)^2$.
\end{proof}

A further ingredient in the proof of
Theorem~\ref{thm:limit-stationary} is the lower semicontinuity of the
stationarity measure with respect to the $W_2$ topology. We record this
auxiliary fact here and defer its proof to the appendix.

\begin{proposition}[Lower semicontinuity of the stationarity measure]
\label{prop:gap-lsc}
The map $\rho\mapsto m_{\Omega}(\rho)$ is lower semicontinuous on
$\mathcal{P}_2(\Omega)$ with respect to the $W_2$ topology.
\end{proposition}

\begin{theorem}[Asymptotic first-order properties of the constrained flow]
\label{thm:limit-stationary}
Let $\Omega\subset\R^d$ be nonempty, compact, and convex, and let
$f\in C^1(\Omega;\R^m)$. Suppose that
$(\rho_t)_{t\ge 0}\subset\mathcal{P}_2(\Omega)$ is an absolutely continuous
curve satisfying \eqref{eq:flow_bar}. Then:

\noindent
(i) the dissipation is integrable, in the sense that
$\int_0^\infty \|v_t\|_{\rho_t}^2\,dt<\infty$;

\noindent
(ii) there exists a sequence $t_k\to\infty$ such that
$\|v_{t_k}\|_{\rho_{t_k}}\to 0$;

\noindent
(iii) every $W_2$-limit point of $(\rho_t)_{t\ge 0}$ is stationary in the
sense of Definition~\ref{def:stationary}.
\end{theorem}

\begin{proof}
For (i), Proposition~\ref{prop:energy-identity-nonsmooth} gives
\[
\mathcal{F}_E(\rho_T)-\mathcal{F}_E(\rho_0)
=
\int_0^T \|v_t\|_{\rho_t}^2\,dt,
\qquad T\ge 0.
\]
Since $\Omega$ is compact and $f$ is continuous,
$0 \le \mathcal{F}_E(\rho) \le \sup_{x\in\Omega}\|f(x)\|^2$ for every
$\rho\in\mathcal{P}_2(\Omega)$, so the left-hand side is bounded uniformly
in $T$. Letting $T\to\infty$ yields
$\int_0^\infty \|v_t\|_{\rho_t}^2\,dt<\infty$.

Part (ii) follows directly: if
$\liminf_{t\to\infty}\|v_t\|_{\rho_t}>0$, then the integrand is bounded
away from zero on a set of infinite measure, contradicting (i).

For (iii), let $\bar\rho$ be any $W_2$-limit point of
$(\rho_t)_{t\ge0}$, and choose $s_k\to\infty$ such that
$\rho_{s_k}\to\bar\rho$ in $W_2$. By (i),
$\int_{s_k}^{s_k+1}\|v_t\|_{\rho_t}^2\,dt\to0$, so one may select
$t_k\in[s_k,s_k+1]$ with
$\|v_{t_k}\|_{\rho_{t_k}}^2 \le
\int_{s_k}^{s_k+1}\|v_t\|_{\rho_t}^2\,dt$, giving
$\|v_{t_k}\|_{\rho_{t_k}}\to0$. The dynamical characterization of $W_2$
\citep[see, e.g.,][Section~8.3]{ambrosio_gradient_2008} and the bound
$t_k - s_k \le 1$ yield
\[
W_2^2(\rho_{s_k},\rho_{t_k})
\le
(t_k-s_k)\int_{s_k}^{t_k}\|v_t\|_{\rho_t}^2\,dt
\le
\int_{s_k}^{s_k+1}\|v_t\|_{\rho_t}^2\,dt\to0,
\]
so $\rho_{t_k}\to\bar\rho$ in $W_2$. Since
$m_\Omega(\rho_{t_k}) = \|v_{t_k}\|_{\rho_{t_k}} \to 0$,
Proposition~\ref{prop:gap-lsc} gives
$m_{\Omega}(\bar\rho) \le
\liminf_{k\to\infty}m_{\Omega}(\rho_{t_k}) = 0$. Hence
$\bar{\nabla}_{W_2}\mathcal{F}_E(\bar\rho) = 0$, and
Proposition~\ref{prop:gap-stationary} implies that $\bar\rho$ is
stationary.
\end{proof}

\begin{remark}
\label{rem:averaging}
Proposition~\ref{prop:energy-identity-nonsmooth} immediately yields the
averaging estimate
\[
\inf_{0\le t\le T}\|v_t\|_{\rho_t}^2
\le
\frac{\sup_{\mu\in\mathcal{P}_2(\Omega)}\mathcal{F}_E(\mu)
-\mathcal{F}_E(\rho_0)}{T}
=
O(T^{-1}),
\qquad T>0.
\]
In particular, the minimal squared speed over $[0,T]$ decays at least at
rate $T^{-1}$.
\end{remark}

\begin{remark}
\label{rem:discussion}
Together, Proposition~\ref{prop:energy-identity-nonsmooth} and
Theorem~\ref{thm:limit-stationary} provide a complete limit-point
stationarity theory for the constrained steepest-ascent flow despite the
nonsmoothness of the $E$-criterion: the flow satisfies an exact energy
identity, has finite total dissipation, and every limit point is stationary. In the simple-eigenvalue case,
Remark~\ref{rem:simple-eigenvalue} shows that the constrained Wasserstein
steepest-ascent field reduces to the projected Wasserstein gradient, so the
present framework recovers the smooth constrained flow as a special case. By
comparison, the smooth unconstrained flows for $A$- and $D$-optimality
studied by \citet{jin_optimal_2024} did not come with a corresponding
asymptotic convergence guarantee.
\end{remark}

\subsection{Particle computation and semidefinite reduction}
\label{sec:particle-computation}

The constructions in Sections~\ref{sec:steepest-ascent}
and~\ref{sec:flow} are formulated for general measures in
$\mathcal{P}_2(\Omega)$. To obtain a practical algorithm, we now specialize
to empirical measures and show that the constrained Wasserstein
steepest-ascent direction computation reduces to a low-dimensional
semidefinite programme.

For an empirical measure $\rho_N=N^{-1}\sum_{i=1}^N\delta_{x_i}$, write
$M_N := M_{\rho_N}$ for the corresponding information matrix. We seek a
particle realization of the constrained Wasserstein steepest-ascent
direction. Identifying a feasible particle direction with a vector
$v=(v_1,\ldots,v_N)\in(\R^d)^N$ satisfying $v_i\in T_{\Omega}(x_i)$, the
particle analogue of the constrained steepest-ascent direction is
\begin{equation}
\label{eq:particle-direction-problem}
v_N^\star\in\arg\max\left\{
D\mathcal{F}_E(\rho_N)[v]:
\|v\|_{\rho_N}\le 1,\;
v_i\in T_{\Omega}(x_i),\ i=1,\ldots,N
\right\}.
\end{equation}
This problem has $Nd$ optimization variables, which is prohibitive when $N$
or $d$ is large. However, the eigenspace structure of $M_N$ permits a
reduction to a convex problem whose dimension depends only on the
multiplicity $s_N$ of $\lambda_{\min}(M_N)$.

Let $V_N\in\R^{m\times s_N}$ have orthonormal columns spanning the
eigenspace of $M_N$ associated with $\lambda_{\min}(M_N)$. By
Proposition~\ref{prop:G-structure}, every $G\in \mathcal{G}(M_N)$ takes the
form $G=V_NSV_N^\T$ with $S\succeq 0$ and $\tr(S)=1$. The key idea is to
parametrize the steepest-ascent field by $S$ rather than by the particle
directions directly. By Proposition~\ref{prop:gap-representation}, the
steepest-ascent field at $\rho_N$ is
$\bar{\nabla}_{W_2}\mathcal{F}_E(\rho_N) =
\Pi_{T_{\Omega}(\cdot)}\bigl(a_{G_{\rho_N}^\star}(\cdot)\bigr)$, so
evaluating this field at each particle $x_i$ with $G=V_NSV_N^\T$ gives the
parametrized particle directions
\begin{equation}
\label{eq:diS}
v_i(S)
:=
\Pi_{T_{\Omega}(x_i)}
\bigl(
2\nabla f(x_i)^\T V_NSV_N^\T f(x_i)
\bigr).
\end{equation}
The optimal $S$ is then determined by a finite-dimensional minimization.

\begin{proposition}[Low-dimensional convex formulation]
\label{prop:particle-convex}
Let
\begin{equation}
\label{eq:particle-direction-sdp}
S_N^\star
\in
\arg\min_{S\succeq 0,\ \tr(S)=1}
\left(
\frac1N\sum_{i=1}^N \|v_i(S)\|^2
\right)^{1/2}.
\end{equation}
Then any minimizer $S_N^\star$ induces the particle realization of the
constrained Wasserstein steepest-ascent field at $\rho_N$, namely
\[
\bigl(v_1(S_N^\star),\ldots,v_N(S_N^\star)\bigr).
\]
Moreover, the optimization problem \eqref{eq:particle-direction-sdp} is
finite-dimensional and convex.
\end{proposition}

A proof is given in the appendix. The feasible set
$\{S\in\mathbb{S}^{s_N}: S\succeq 0,\ \tr(S)=1\}$ is an
$s_N\times s_N$ matrix spectrahedron, so
\eqref{eq:particle-direction-sdp} is a semidefinite programme with
dimension determined by the multiplicity $s_N$ of $\lambda_{\min}(M_N)$
rather than the $Nd$ particle variables in the original problem
\eqref{eq:particle-direction-problem}. In our implementation, this
programme is solved using SDPT3
\citep{toh_sdpt3_1999, tutuncu_solving_2003}.

The particle stationarity measure
\[
\widehat m_N(x_1,\ldots,x_N)
:=
\left(
\frac1N\sum_{i=1}^N \|v_i(S_N^\star)\|^2
\right)^{1/2}
\]
is the particle-level counterpart of $m_\Omega(\rho)$ and serves as the
stopping criterion in Algorithm~\ref{alg:particle-ascent}.

\begin{algorithm}[t]
\caption{Projected particle constrained Wasserstein steepest ascent}
\label{alg:particle-ascent}
\KwIn{Initial particles $x_1^0,\ldots,x_N^0\in\Omega$, step sizes
$\alpha_k>0$, tolerance $\varepsilon>0$}
\KwOut{Empirical design
$\rho_N^K=N^{-1}\sum_{i=1}^N\delta_{x_i^K}$}
\For{$k=0,1,2,\ldots$}{
  Form $M_k=N^{-1}\sum_{i=1}^N f(x_i^k)f(x_i^k)^\T$ and compute an
  orthonormal basis $V_k$ for the eigenspace of $\lambda_{\min}(M_k)$\;
  Solve \eqref{eq:particle-direction-sdp} to obtain $S_k^\star$\;
  Compute $v_{i,k}^\star=v_i(S_k^\star)$ via \eqref{eq:diS} for
  $i=1,\ldots,N$ and the stationarity measure $\widehat m_N(x^k)$\;
  \If{$\widehat m_N(x^k)\le\varepsilon$}{
    stop\;
  }
  Update $x_i^{k+1}=\Pi_\Omega(x_i^k+\alpha_k v_{i,k}^\star)$ for
  $i=1,\ldots,N$\;
}
\end{algorithm}

The dominant per-iteration costs are the evaluation of $f$ and $\nabla f$
at the $N$ particles, the assembly of $M_N$, and its eigendecomposition.
The semidefinite programme \eqref{eq:particle-direction-sdp} involves only
an $s_N\times s_N$ matrix variable and is typically inexpensive relative to
these costs.

\subsection{Extension to minimax designs with respect to the single parameters}
\label{sec:scaled-param}

The main development above centres on the $E$-criterion
$\mathcal{F}_E(\rho)=\lambda_{\min}(M_\rho)$, but the constrained
Wasserstein steepest-ascent framework is not specific to this criterion.
Its essential requirement is that the design criterion be a pointwise
minimum of smooth functionals, a structure shared by many classical minimax
criteria in optimal design.  To illustrate this generality, we consider the
problem of minimizing the largest diagonal element of $M_\rho^{-1}$, a
criterion studied by \citet{murty_minimax_1971} and
\citet{wong_unified_1992}.

We follow the classical formulation in which the parameterization has
already been chosen so that the individual parameters are appropriately
scaled.  Under this parameterization, the minimax-single-parameter
criterion is naturally written using the standard basis vectors
$e_1,\ldots,e_m$ of $\R^m$.  The problem is
\begin{equation}
\label{eq:sp-minmax}
\min_{\rho\in\mathcal{P}_2(\Omega)}
\max_{1\le j\le m}[M_\rho^{-1}]_{jj},
\end{equation}
which seeks a design minimizing the largest variance among all individual
parameter estimates.  Defining
$\mathcal{F}_{\mathrm{sp}}(\rho)
:=\min_{1\le j\le m}\{-[M_\rho^{-1}]_{jj}\}$,
the problem \eqref{eq:sp-minmax} becomes
$\max_{\rho\in\mathcal{P}_2(\Omega)}\mathcal{F}_{\mathrm{sp}}(\rho)$.
Like $\mathcal{F}_E$, the functional $\mathcal{F}_{\mathrm{sp}}$ is a
pointwise minimum of smooth functionals, so the constrained Wasserstein
steepest-ascent framework of
Sections~\ref{sec:directional-derivative}--\ref{sec:steepest-ascent}
applies.

For $\rho\in\mathcal{P}_2(\Omega)$ with nonsingular $M_\rho$, define the
active set
$$\mathcal{A}_{\mathrm{sp}}(\rho)
:=\arg\max_{1\le j\le m}[M_\rho^{-1}]_{jj},$$
the set of indices attaining the largest diagonal element of
$M_\rho^{-1}$.  Let
$(\rho_t)_{|t|<\varepsilon}\subset\mathcal{P}_2(\Omega)$ be an absolutely
continuous curve satisfying
$\partial_t\rho_t+\nabla\!\cdot(\rho_tu_t)=0$ with
$u_t(x)\in T_\Omega(x)$ for $\rho_t$-almost every $x$ and almost every
$t$.  For each fixed $j$, the map $\rho\mapsto [M_\rho^{-1}]_{jj}$ is
smooth whenever $M_\rho$ is nonsingular, and differentiation of
$M_{\rho_t}^{-1}$ along the flow gives, for almost every $t$,
\[
\frac{d}{dt}\bigl(-[M_{\rho_t}^{-1}]_{jj}\bigr)
=
\bigl\langle
2\nabla f(\cdot)^\T M_{\rho_t}^{-1}e_je_j^\T M_{\rho_t}^{-1}f(\cdot),\,
u_t
\bigr\rangle_{\rho_t}.
\]
Since $\mathcal{F}_{\mathrm{sp}}=\min_{1\le j\le m}(-[M_\rho^{-1}]_{jj})$
is the pointwise minimum of finitely many smooth functionals, Danskin's
theorem \citep[Prop.~B.25]{Bertsekas1999} gives
$\frac{d}{dt}\mathcal{F}_{\mathrm{sp}}(\rho_t)
=D\mathcal{F}_{\mathrm{sp}}(\rho_t)[u_t]$
for almost every $t$,
where the directional derivative is defined by
\begin{equation}
\label{eq:sp-directional}
D\mathcal{F}_{\mathrm{sp}}(\rho)[u]
:=
\min_{j\in\mathcal{A}_{\mathrm{sp}}(\rho)}
\bigl\langle
2\nabla f(\cdot)^\T M_\rho^{-1}e_je_j^\T M_\rho^{-1}f(\cdot),\,
u
\bigr\rangle_\rho.
\end{equation}

With this directional derivative, the constrained Wasserstein
steepest-ascent construction of Section~\ref{sec:steepest-ascent} carries
over directly.  The steepest-ascent direction and field are defined exactly
as in Definition~\ref{def:steepest-ascent}, with
$D\mathcal{F}_E(\rho)[u]$ replaced by
$D\mathcal{F}_{\mathrm{sp}}(\rho)[u]$ throughout: one maximizes
$D\mathcal{F}_{\mathrm{sp}}(\rho)[u]$ over the feasible unit ball $K_\rho$
to obtain the steepest-ascent direction, and the steepest-ascent field is
the product of this direction with the attained maximum.  A measure $\rho$
is stationary for \eqref{eq:sp-minmax} when no feasible direction can
increase $\mathcal{F}_{\mathrm{sp}}$.  The corresponding notions of
stationarity and particle-level implementation can be formulated
analogously.  The associated flow analysis is technically simpler than for
the $E$-criterion, because the nonsmoothness here is of finite-max type
rather than spectral.  We do not develop a separate parallel convergence
theory here.

The following analogue of Proposition~\ref{prop:gap-representation} gives an
explicit representation of the steepest-ascent field.  In contrast to the
$E$-criterion, where the particle-level direction subproblem is a
semidefinite programme over a spectrahedron, here the auxiliary minimization
is a convex programme over a probability simplex.

\begin{proposition}[Steepest-ascent field for $\mathcal{F}_{\mathrm{sp}}$]
\label{prop:sp-field}
For every $\rho\in\mathcal{P}_2(\Omega)$ with nonsingular $M_\rho$, there
exist weights $\alpha_j^\star\ge 0$ for
$j\in\mathcal{A}_{\mathrm{sp}}(\rho)$ with
$\sum_{j\in\mathcal{A}_{\mathrm{sp}}(\rho)}\alpha_j^\star=1$ such that the
constrained Wasserstein steepest-ascent field for
$\mathcal{F}_{\mathrm{sp}}$ is
\begin{equation}
\label{eq:sp-field}
\bar{\nabla}_{W_2}\mathcal{F}_{\mathrm{sp}}(\rho)(x)
=
\Pi_{T_\Omega(x)}\biggl(
2\nabla f(x)^\T M_\rho^{-1}
\Bigl(\sum_{j\in\mathcal{A}_{\mathrm{sp}}(\rho)}
\alpha_j^\star\,e_je_j^\T\Bigr)
M_\rho^{-1}f(x)\biggr),
\qquad\textit{$\rho$-a.e.\ }x,
\end{equation}
where $\alpha^\star$ minimizes
\begin{equation}
\label{eq:sp-simplex}
\biggl\|\Pi_{T_\Omega(\cdot)}\biggl(
2\nabla f(\cdot)^\T M_\rho^{-1}
\Bigl(\sum_{j\in\mathcal{A}_{\mathrm{sp}}(\rho)}
\alpha_j\,e_je_j^\T\Bigr)
M_\rho^{-1}f(\cdot)\biggr)\biggr\|_\rho
\end{equation}
over all $\alpha_j\ge 0$ with
$\sum_{j\in\mathcal{A}_{\mathrm{sp}}(\rho)}\alpha_j=1$.
\end{proposition}

\begin{proof}
Since $\mathcal{A}_{\mathrm{sp}}(\rho)$ is finite, the minimum over
$j\in\mathcal{A}_{\mathrm{sp}}(\rho)$ in \eqref{eq:sp-directional} can be
rewritten as a minimum over convex combinations:
\[
D\mathcal{F}_{\mathrm{sp}}(\rho)[u]
=
\min_{\substack{\alpha_j\ge 0,\\
\sum_{j\in\mathcal{A}_{\mathrm{sp}}(\rho)}\alpha_j=1}}
\sum_{j\in\mathcal{A}_{\mathrm{sp}}(\rho)}
\alpha_j
\bigl\langle
2\nabla f(\cdot)^\T M_\rho^{-1}e_je_j^\T M_\rho^{-1}f(\cdot),\,u
\bigr\rangle_\rho.
\]
Applying the same maximin interchange as in the proof of
Proposition~\ref{prop:gap-representation}, with the simplex of active-set
weights replacing the spectrahedron there, we obtain
\[
\sup_{u\in K_\rho}D\mathcal{F}_{\mathrm{sp}}(\rho)[u]
=
\min_{\substack{\alpha_j\ge 0,\\
\sum_{j\in\mathcal{A}_{\mathrm{sp}}(\rho)}\alpha_j=1}}
\sup_{u\in K_\rho}
\bigl\langle
2\nabla f(\cdot)^\T M_\rho^{-1}
\Bigl(\sum_{j\in\mathcal{A}_{\mathrm{sp}}(\rho)}
\alpha_j\,e_je_j^\T\Bigr)
M_\rho^{-1}f(\cdot),\,u
\bigr\rangle_\rho.
\]
For each fixed $\alpha$, the inner supremum over $K_\rho$ equals the
$L^2(\rho)$ norm of the tangent-cone projection of the corresponding vector
field, by the same Moreau decomposition argument as in the proof of
Proposition~\ref{prop:gap-representation}.  Taking the outer minimum over
$\alpha$ yields \eqref{eq:sp-field}.
\end{proof}

For an empirical measure $\rho_N=N^{-1}\sum_{i=1}^N\delta_{x_i}$, the
direction computation reduces to minimizing \eqref{eq:sp-simplex} at
$\rho=\rho_N$. The auxiliary variable $\alpha$ lies in a probability
simplex whose dimension equals the cardinality of
$\mathcal{A}_{\mathrm{sp}}(\rho_N)$. A numerical illustration is given
in Section~\ref{sec:sp-experiment}, where the method is applied to a
minimax-single-parameter design problem with a known theoretical
optimum.

\section{Numerical experiments}
\label{sec:experiments}

\subsection{Setup}

We assess the proposed methods on several optimal design problems and
compare them with the basic version of particle swarm optimization
\citep{kennedy1995particle}, which has become a standard baseline for
computing optimal designs in moderate to high dimensions
\citep{wire,zack1}. For the smooth $D$-criterion we use the particle
Wasserstein gradient flow of Corollary~\ref{cor:particle_lift}; for the
nonsmooth $E$-criterion and the minimax-single-parameter criterion of
Section~\ref{sec:scaled-param}, we use the constrained Wasserstein
steepest-ascent method of Algorithm~\ref{alg:particle-ascent}. The test
problems are the second-order response surface model
\begin{equation}
\label{eq:so}
f(x) = \Bigl(1,\, x_1,\ldots,x_k,\, x_1^2,\ldots,x_k^2,\, x_1x_2,\ldots,x_{k-1}x_k\Bigr)^\T
\end{equation}
with $k\in\{2,5\}$ on the unit cube $[-1,1]^k$ and the unit ball
$\{x\in\R^k:\|x\|\le 1\}$, together with a seven-dimensional logistic
regression model on $[-3,3]^7$. All runs use $N=100$ particles
initialized uniformly on $\Omega$; for the $E$-criterion, the direction
subproblem~\eqref{eq:particle-direction-sdp} is solved using SDPT3
\citep{toh_sdpt3_1999,tutuncu_solving_2003}. Particle swarm optimization maintains a swarm of candidate designs whose positions are updated at each iteration by a stochastic combination of each member's best past position and the swarm's global best, weighted by cognitive and social coefficients and an inertia term carrying over the previous velocity. We implement it in MATLAB by adapting the publicly available code of \citet{biswas2014pso}; each swarm member encodes an equal-weight $N$-point design, the swarm size is $100$, the iteration budget is $1000$, and the results are summarized over $100$ independent runs. Full update rules and parameter values are given in the appendix.  Table~\ref{tab:results} reports the objective
values. MATLAB codes for reproducing the numerical results in this paper are available at \texttt{https://github.com/Jieling-Shi/wgf-optimal-design-matlab}.

\begin{table}[t]
\centering
\caption{Objective values for optimal design problems under various criteria and design regions.}
\small
\label{tab:results}
\begin{tabular}{llcrrrr}
Criterion & Case & WGF/CWSA & \multicolumn{3}{c}{PSO} & Optimal \\
\cline{4-6}
          &      &          & Best & Mean & Worst & \\
$D$ & SO($k=5$), cube & $-14.716$ & $-14.845$ & $-19.564$ & $-25.471$ & $-14.270$ \\
$D$ & SO($k=5$), ball & $-60.681$ & $-60.695$ & $-61.019$ & $-62.808$ & $-60.680$ \\
$E$ & SO($k=2$), cube & $0.200$ & $0.199$ & $0.193$ & $0.183$ & $0.200$ \\
$E$ & SO($k=2$), ball & $0.100$ & $0.099$ & $0.096$ & $0.093$ & $0.100$ \\
$E$ & SO($k=5$), ball & $0.027$ & $0.025$ & $0.024$ & $0.023$ & $0.027$ \\
$E$ & SO($k=5$), cube & $0.192$ & $0.167$ & $0.153$ & $0.062$ & $0.200$ \\
$E$ & Logistic, $[-3,3]^7$ & $0.154$ & $0.153$ & $0.152$ & $0.150$ & \\
SP  & SO($k=5$), cube & $-1.000$ & $-1.045$ & $-1.144$ & $-1.275$ & $-1.000$ \\
\end{tabular}

\smallskip
\noindent\footnotesize WGF, particle Wasserstein gradient flow (Corollary~\ref{cor:particle_lift}), used for the $D$-criterion; CWSA, constrained Wasserstein steepest ascent (Algorithm~\ref{alg:particle-ascent}), used for the $E$- and SP-criteria; PSO, particle swarm optimization (swarm size $100$, results over $100$ independent runs); SO, second-order response surface model; SP, minimax-single-parameter criterion (Section~\ref{sec:scaled-param}), using an orthonormal parameterization. Optimal: theoretical value, when available.
\end{table}

\subsection[D-optimal design]{$D$-optimal design}

The $D$-criterion $\mathcal{F}_D(\rho)=\log\det(M_{\rho})$ is smooth, so
the particle Wasserstein gradient flow of
Corollary~\ref{cor:particle_lift} applies directly. We consider the model \eqref{eq:so} with $k=5$ on the unit cube and ball, for which the
theoretical $D$-optimal values are available from
\citet{kiefer1959optimum}. Figure~\ref{fig:Dopt} shows that the particle
Wasserstein gradient flow enters the near-optimal regime within a small
number of iterations and then stabilizes on both regions, attaining the
theoretical optimum in each case. Particle swarm optimization matches it
in its best runs, but the mean across runs falls below the optimum,
modestly on the ball and substantially on the cube, with correspondingly
larger dispersion; see the $D$-optimal rows of Table~\ref{tab:results}.

\begin{figure}[t]
\centering
\includegraphics[width=0.48\textwidth]{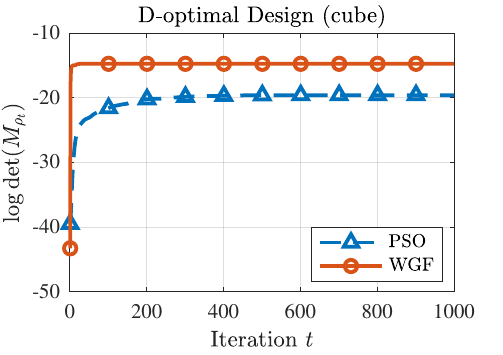}\hfill
\includegraphics[width=0.48\textwidth]{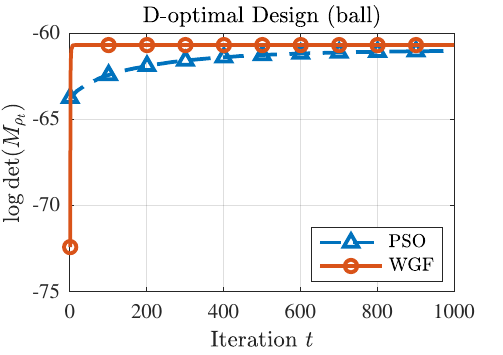}
\caption{Convergence of the particle Wasserstein gradient flow (solid) and particle swarm optimization (dashed, averaged over $100$ runs) for the $D$-optimal design on the unit cube (left) and unit ball (right), with $k=5$.}
\label{fig:Dopt}
\end{figure}

\subsection[E-optimal design]{$E$-optimal design}

\textit{Second-order response surface model.}
We apply Algorithm~\ref{alg:particle-ascent} to the model \eqref{eq:so}
under the $E$-criterion $\mathcal{F}_E(\rho)=\lambda_{\min}(M_{\rho})$,
with $k\in\{2,5\}$ on the unit cube and ball; theoretical benchmarks are
from \citet{dette_e-optimal_2014}. In all four configurations, the
constrained Wasserstein steepest-ascent method enters the near-optimal
regime within a small number of iterations and stabilizes there. At
$k=2$, particle swarm optimization also reaches the theoretical optimum
on both regions, with only moderate dispersion across runs. At $k=5$,
the two methods separate: on the ball, the constrained Wasserstein
steepest-ascent method attains the optimum while the best particle
swarm run falls short of it; on the cube, the constrained Wasserstein
steepest-ascent method remains close to the optimum, whereas even the
best particle swarm run lies well below it and its worst-case value
drops further still. In the $k=5$ cases, particle swarm optimization is
still climbing at the end of the iteration budget without reaching the
plateau; see the $E$-optimal response surface rows of
Table~\ref{tab:results} and Fig.~\ref{fig:E_so}.

\noindent\textit{Logistic regression model.}
We consider the logistic regression model of \citet{stella}, in which
the mean response $\mu(x,\theta)$ satisfies the canonical logit link
\[
\log\left\{\frac{\mu(x,\theta)}{1-\mu(x,\theta)}\right\}=v(x)^\T\theta,
\qquad
v(x)=(1,x_1,\ldots,x_7)^\T,
\]
with nominal parameter
\[
\theta^\ast=(-0.4926,-0.6280,-0.3283,0.4378,0.5283,-0.6120,-0.6837,-0.2061)^\T.
\]
The design space is $[-3,3]^7$; the model has $m=8$ parameters and $d=7$
design variables, and no closed-form optimum is available. A single run
of the constrained Wasserstein steepest-ascent method exceeds the best
of $100$ independent particle swarm runs, which themselves exhibit
substantial dispersion; see Table~\ref{tab:results} and
Fig.~\ref{fig:Eopt-logistic}.

\begin{figure}[t]
\centering
\includegraphics[width=0.48\textwidth]{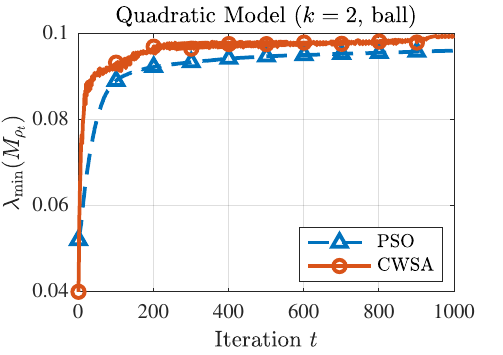}\hfill
\includegraphics[width=0.48\textwidth]{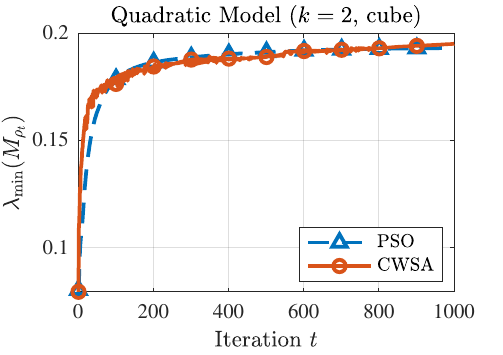}
\vspace{2mm}
\includegraphics[width=0.48\textwidth]{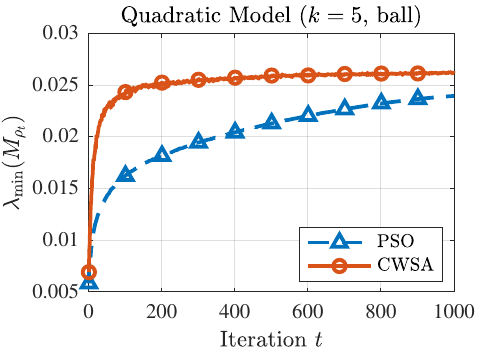}\hfill
\includegraphics[width=0.48\textwidth]{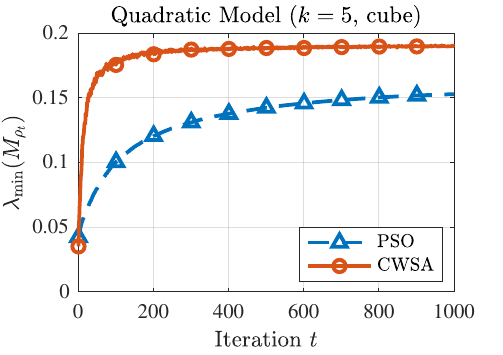}
\caption{Convergence of the constrained Wasserstein steepest-ascent method (solid) and particle swarm optimization (dashed, averaged over $100$ runs) for the $E$-optimal design on the second-order response surface model: $k=2$ (top) and $k=5$ (bottom), on the unit ball (left) and unit cube (right).}
\label{fig:E_so}
\end{figure}

\begin{figure}[t]
\centering
\begin{subfigure}[t]{0.48\textwidth}
  \centering
  \includegraphics[width=\textwidth]{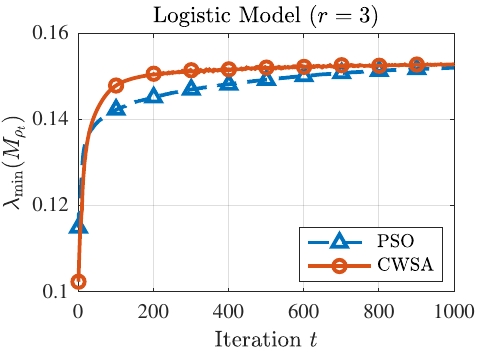}
  \caption{}
  \label{fig:Eopt-logistic}
\end{subfigure}\hfill
\begin{subfigure}[t]{0.48\textwidth}
  \centering
  \includegraphics[width=\textwidth]{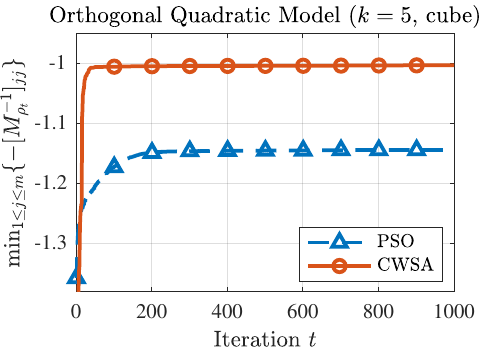}
  \caption{}
  \label{fig:sp}
\end{subfigure}
\caption{Convergence of the constrained Wasserstein steepest-ascent method (solid) and particle swarm optimization (dashed, averaged over $100$ runs). (a) $E$-optimal design in the logistic regression model on $[-3,3]^7$; the constrained Wasserstein steepest-ascent curve is smoothed by averaging over every $10$ iterations. (b) Minimax-single-parameter design on the second-order response surface model with $k=5$ on the unit cube, using an orthonormal parameterization.}
\label{fig:E-sp}
\end{figure}

\subsection{Minimax-single-parameter design}
\label{sec:sp-experiment}

\citet{murty_minimax_1971}  proposed the minimax-single-parameter
problem,
\[
\min_{\rho\in\mathcal{P}_2(\Omega)}\,\max_{1\le j\le m}[M_\rho^{-1}]_{jj},
\]
and, using a standard elementary minimax result of game theory,
derived a sufficient condition for optimality and obtained optimal
designs for polynomial regression models by verifying it analytically.
The minimax-single-parameter criterion is not
invariant under general rescaling of the model parameters, so we adopt
an orthonormal basis under the uniform measure on $[-1,1]^5$. For
$k=5$, this basis consists of the $m=21$ functions $1$,
$\sqrt{3}\,x_i$ ($i=1,\ldots,5$), $\sqrt{5}\,(3x_i^2-1)/2$
($i=1,\ldots,5$) and $3x_ix_j$ ($i<j$); we denote the resulting
regression vector by $\tilde f(x)$ and write
$M_\rho=\int_\Omega \tilde f(x)\tilde f(x)^\T\rho(dx)$ for the
corresponding information matrix. With this choice of basis, the
uniform measure on the cube gives $M_\rho=I_m$.

To align with the ascent formulation, we work with
$\mathcal{F}_{\mathrm{sp}}(\rho)=\min_{1\le j\le m}\bigl\{-[M_\rho^{-1}]_{jj}\bigr\}$
and equivalently solve
$\max_{\rho\in\mathcal{P}_2(\Omega)}\mathcal{F}_{\mathrm{sp}}(\rho)$,
as introduced in Section~\ref{sec:scaled-param}. Since the first
component of $\tilde f$ is the constant function $1$,
$[M_\rho]_{11}=1$ for every $\rho\in\mathcal{P}_2(\Omega)$. As
$[M^{-1}]_{jj}\ge 1/[M]_{jj}$ for any positive definite $M$, this gives
$[M_\rho^{-1}]_{11}\ge 1$ and hence $\mathcal{F}_{\mathrm{sp}}(\rho)\le
-1$. The uniform measure on $\Omega$ gives $M_\rho=I_m$ and attains
this bound, which is therefore the theoretical optimum.

The constrained Wasserstein steepest-ascent method enters the
near-optimal regime within a small number of iterations and stabilizes
at the theoretical optimum, whereas particle swarm optimization falls
short of it and its mean across runs deteriorates markedly; see the
last row of Table~\ref{tab:results} and Fig.~\ref{fig:sp}. The pattern
matches the higher-dimensional $E$-optimal experiments and indicates
that the reliability of the constrained Wasserstein steepest-ascent
method relative to particle swarm optimization is not specific to the
$E$-criterion. The proposed method does not rely on model-specific
sufficient conditions of the type used by \citet{murty_minimax_1971}
and requires only continuous differentiability of the regression
vector on $\Omega$, so it extends directly beyond the polynomial
setting of that work.

\section{Discussion}\label{sec:discussion}

The constrained Wasserstein steepest-ascent framework developed here
provides a principled gradient-based approach for finding $E$-optimal
designs despite the nonsmoothness of the $\lambda_{\min}$ criterion. The
energy identity and limit-point stationarity theorem provide the first
rigorous asymptotic guarantee for a Wasserstein-based method applied to
$E$-optimal design problems, and the semidefinite reduction yields a
practical algorithm by confining the direction computation to a matrix
variable whose dimension equals the multiplicity of the smallest
eigenvalue. The numerical experiments confirm that the method is especially
effective in higher-dimensional settings where the basic version of
particle swarm optimization becomes less reliable. The extension to the
minimax-single-parameter criterion and the accompanying numerical
experiments demonstrate that the framework is not confined to the
$E$-optimality criterion and applies to other design problems with a
nonsmooth minimax objective as well.

\section*{Acknowledgements}
The research of Tong was partially supported by the Ministry of Education, Singapore, under Tier 1 Grant A-8002956-00-00. The research of Wong was partially supported by the Yushan Fellow Program of the Ministry of Education, Taiwan (Grant MOE-108-YSFMS-0004-012-P1).

\clearpage
\appendix

\makeatletter
\@addtoreset{equation}{section}
\@addtoreset{figure}{section}
\@addtoreset{table}{section}
\@addtoreset{theorem}{section}
\@addtoreset{lemma}{section}
\@addtoreset{corollary}{section}
\@addtoreset{proposition}{section}
\@addtoreset{definition}{section}
\@addtoreset{assumption}{section}
\@addtoreset{remark}{section}
\renewcommand{\theequation}{\thesection.\arabic{equation}}
\renewcommand{\thefigure}{\thesection.\arabic{figure}}
\renewcommand{\thetable}{\thesection.\arabic{table}}
\renewcommand{\thetheorem}{\thesection.\arabic{theorem}}
\renewcommand{\thelemma}{\thesection.\arabic{lemma}}
\renewcommand{\thecorollary}{\thesection.\arabic{corollary}}
\renewcommand{\theproposition}{\thesection.\arabic{proposition}}
\renewcommand{\thedefinition}{\thesection.\arabic{definition}}
\renewcommand{\theassumption}{\thesection.\arabic{assumption}}
\renewcommand{\theremark}{\thesection.\arabic{remark}}
\renewcommand{\theHsection}{appendix.\Alph{section}}
\renewcommand{\theHsubsection}{appendix.\Alph{section}.\arabic{subsection}}
\renewcommand{\theHequation}{appendix.\Alph{section}.\arabic{equation}}
\renewcommand{\theHfigure}{appendix.\Alph{section}.\arabic{figure}}
\renewcommand{\theHtable}{appendix.\Alph{section}.\arabic{table}}
\renewcommand{\theHtheorem}{appendix.\Alph{section}.\arabic{theorem}}
\renewcommand{\theHlemma}{appendix.\Alph{section}.\arabic{lemma}}
\renewcommand{\theHcorollary}{appendix.\Alph{section}.\arabic{corollary}}
\renewcommand{\theHproposition}{appendix.\Alph{section}.\arabic{proposition}}
\renewcommand{\theHdefinition}{appendix.\Alph{section}.\arabic{definition}}
\renewcommand{\theHassumption}{appendix.\Alph{section}.\arabic{assumption}}
\renewcommand{\theHremark}{appendix.\Alph{section}.\arabic{remark}}
\makeatother

\section*{Appendices}

These appendices contain proofs of results stated in Sections~\ref{sec:smooth} and~\ref{sec:nonsmooth} of the main paper, together with implementation details for the particle swarm optimization benchmark of Section~\ref{sec:experiments}. Throughout, the notation is consistent with that used in the main text.

\section{Additional preliminaries}

\subsection{Pushforward perturbations and Wasserstein first variations}
The following characterization is a reformulation of standard first-variation results in Wasserstein space; see, for example, \citep[Ch.~8]{ambrosio_gradient_2008} and \citep[\S~5.5]{santambrogio2015ot}.

\begin{lemma}[Pushforward characterization of the Wasserstein gradient]
\label{lem:pushforward-characterization}
Let $\mathcal F:\mathcal P_2(\R^d)\to\R$ be Wasserstein differentiable at
$\rho\in\mathcal P_2(\R^d)$ in the sense of Definition~\ref{def:Wassersteingrad},
and let
\[
G\in \operatorname{Tan}_{\rho}\mathcal P_2(\R^d).
\]
Then the following are equivalent:
\begin{enumerate}
\item[(i)] $G=\nabla_{W_2}\mathcal F(\rho)$ in
$L^2(\rho;\R^d)$;
\item[(ii)] for every $v\in \operatorname{Tan}_{\rho}\mathcal P_2(\R^d)$ and every sufficiently small $\varepsilon\in\R$, if
\[
T_\varepsilon(x):=x+\varepsilon v(x),
\qquad
\rho^\varepsilon:=(T_\varepsilon)_\#\rho,
\]
then
\begin{equation}
\label{eq:pushforward-expansion}
\mathcal F(\rho^\varepsilon)-\mathcal F(\rho)
=
\varepsilon
\int_{\R^d}
\langle G(x),v(x)\rangle\,\rho(dx)
+
o\!\bigl(|\varepsilon|\,\|v\|_{L^2(\rho)}\bigr),
\qquad
\varepsilon\to0.
\end{equation}
\end{enumerate}
In particular, $\nabla_{W_2}\mathcal F(\rho)$ is the unique element of
$\operatorname{Tan}_{\rho}\mathcal P_2(\R^d)$ satisfying
\eqref{eq:pushforward-expansion}.
\end{lemma}

\begin{proof}
Let
\[
H:=\nabla_{W_2}\mathcal F(\rho)\in \operatorname{Tan}_{\rho}\mathcal P_2(\R^d).
\]

We first prove (i)$\Rightarrow$(ii). Assume that $G=H$, and fix
$v\in \operatorname{Tan}_{\rho}\mathcal P_2(\R^d)$. If $v=0$, then
$\rho^\varepsilon=\rho$ for all $\varepsilon$, and
\eqref{eq:pushforward-expansion} is trivial. Hence suppose $v\neq 0$.
Set
\[
T_\varepsilon:=\Id+\varepsilon v,
\qquad
\rho^\varepsilon:=(T_\varepsilon)_\#\rho.
\]
Since $v\in L^2(\rho;\R^d)$, we have $\rho^\varepsilon\in \mathcal P_2(\R^d)$
for all sufficiently small $\varepsilon$. Moreover, because
$v\in \operatorname{Tan}_{\rho}\mathcal P_2(\R^d)$, the curve
$\varepsilon\mapsto \rho^\varepsilon$ has initial velocity $v$ at $\rho$; more precisely,
\begin{equation}
\label{eq:tangent-curve-metric-speed}
W_2(\rho,\rho^\varepsilon)
=
|\varepsilon|\,\|v\|_{L^2(\rho)}
+
o(|\varepsilon|),
\qquad
\varepsilon\to0,
\end{equation}
and, as $\varepsilon\to0$,
\begin{equation}
\label{eq:tangent-curve-optimal-map}
\frac{T_\rho^{\rho^\varepsilon}-\Id}{\varepsilon}
\to
v
\end{equation}
in $L^2(\rho;\R^d)$, where $T_\rho^{\rho^\varepsilon}$ denotes the optimal transport map from
$\rho$ to $\rho^\varepsilon$; see, for example,
\citet[Ch.~8]{ambrosio_gradient_2008} and
\citet[\S~5.5]{santambrogio2015ot}.

By Wasserstein differentiability of $\mathcal F$ at $\rho$ in the sense of
Definition~\ref{def:Wassersteingrad},
\[
\mathcal F(\rho^\varepsilon)-\mathcal F(\rho)
=
\int_{\R^d}
\bigl\langle
H(x),\,T_\rho^{\rho^\varepsilon}(x)-x
\bigr\rangle
\,\rho(dx)
+
o\!\bigl(W_2(\rho,\rho^\varepsilon)\bigr),
\qquad
\varepsilon\to0.
\]
Using \eqref{eq:tangent-curve-optimal-map} and $H\in L^2(\rho;\R^d)$, we obtain
\[
\int_{\R^d}
\bigl\langle
H(x),\,T_\rho^{\rho^\varepsilon}(x)-x
\bigr\rangle
\,\rho(dx)
=
\varepsilon
\int_{\R^d}
\langle H(x),v(x)\rangle\,\rho(dx)
+
o(|\varepsilon|).
\]
Since $\|v\|_{L^2(\rho)}>0$ is fixed, this can be rewritten as
\[
\int_{\R^d}
\bigl\langle
H(x),\,T_\rho^{\rho^\varepsilon}(x)-x
\bigr\rangle
\,\rho(dx)
=
\varepsilon
\int_{\R^d}
\langle H(x),v(x)\rangle\,\rho(dx)
+
o\!\bigl(|\varepsilon|\,\|v\|_{L^2(\rho)}\bigr).
\]
Likewise, \eqref{eq:tangent-curve-metric-speed} yields
\[
o\!\bigl(W_2(\rho,\rho^\varepsilon)\bigr)
=
o\!\bigl(|\varepsilon|\,\|v\|_{L^2(\rho)}\bigr).
\]
Substituting the last two relations into the differentiability expansion gives
\eqref{eq:pushforward-expansion}. This proves (ii).

We now prove (ii)$\Rightarrow$(i). Assume that $G\in
\operatorname{Tan}_{\rho}\mathcal P_2(\R^d)$ satisfies
\eqref{eq:pushforward-expansion} for every
$v\in \operatorname{Tan}_{\rho}\mathcal P_2(\R^d)$. By the first part of the proof,
$H=\nabla_{W_2}\mathcal F(\rho)$ also satisfies
\eqref{eq:pushforward-expansion}. Subtracting the two expansions yields
\[
\varepsilon
\int_{\R^d}
\langle G(x)-H(x),v(x)\rangle\,\rho(dx)
=
o\!\bigl(|\varepsilon|\,\|v\|_{L^2(\rho)}\bigr)
\]
for every $v\in \operatorname{Tan}_{\rho}\mathcal P_2(\R^d)$.
Taking $\varepsilon>0$, dividing by $\varepsilon$, and letting
$\varepsilon\downarrow0$, we obtain
\[
\int_{\R^d}
\langle G(x)-H(x),v(x)\rangle\,\rho(dx)=0,
\qquad
\forall\,v\in \operatorname{Tan}_{\rho}\mathcal P_2(\R^d).
\]
Since both $G$ and $H$ belong to
$\operatorname{Tan}_{\rho}\mathcal P_2(\R^d)$, we may choose
$v=G-H$ and conclude that
\[
\|G-H\|_{L^2(\rho)}^2=0.
\]
Hence $G=H$ in $L^2(\rho;\R^d)$, which proves (i). The uniqueness statement follows immediately.
\end{proof}

\subsection{Finite-point gradient interpolation}
\begin{lemma}[Prescribed gradients at finitely many points]
\label{lem:compactly_supported_gradient_interpolation}
Let $x_1,\ldots,x_N\in\R^d$ be pairwise distinct, and let
$\psi_1,\ldots,\psi_N\in\R^d$ be arbitrary. Then there exists
$\varphi\in C_c^\infty(\R^d)$ such that
\[
\nabla\varphi(x_i)=\psi_i,
\qquad
i=1,\ldots,N.
\]
\end{lemma}

\begin{proof}
Choose a radial function $\eta\in C_c^\infty(\R^d)$ such that
\[
\eta(0)=1,
\qquad
\operatorname{supp}(\eta)\subset B(0,1).
\]
Since $\eta$ is smooth and radial, we also have $\nabla\eta(0)=0$.
Set
\[
\delta:=\frac12\min_{i\neq j}\|x_i-x_j\|>0,
\]
and choose $\varepsilon_i\in(0,\delta)$ for $i=1,\ldots,N$. Define
\[
\chi_i(x):=\eta\!\left(\frac{x-x_i}{\varepsilon_i}\right),
\qquad
i=1,\ldots,N.
\]
Then each $\chi_i$ belongs to $C_c^\infty(\R^d)$,
\[
\chi_i(x_i)=1,
\qquad
\nabla\chi_i(x_i)=0,
\]
and
\[
\operatorname{supp}(\chi_i)\subset B(x_i,\varepsilon_i)\subset B(x_i,\delta).
\]
Hence the supports are pairwise disjoint. In particular, for $j\neq i$,
\[
\chi_i(x_j)=0,
\qquad
\nabla\chi_i(x_j)=0.
\]

Now define
\[
\varphi(x)
:=
\sum_{i=1}^N
\bigl[\psi_i^\T(x-x_i)\bigr]\chi_i(x).
\]
Clearly $\varphi\in C_c^\infty(\R^d)$. Differentiating gives
\[
\nabla\varphi(x)
=
\sum_{i=1}^N
\Bigl(
\psi_i\,\chi_i(x)
+
\bigl[\psi_i^\T(x-x_i)\bigr]\nabla\chi_i(x)
\Bigr).
\]
Evaluating at $x=x_j$, all terms with $i\neq j$ vanish, while for $i=j$ we use
$\chi_j(x_j)=1$ and $\nabla\chi_j(x_j)=0$ to obtain
\[
\nabla\varphi(x_j)=\psi_j.
\]
This proves the claim.
\end{proof}

\subsection[Continuity of the minimizer set G(M)]{Continuity of the minimizer set \(\mathcal G(M)\)}
\begin{lemma}[Closed-graph property of the minimizer set]
\label{lem:closed-graph-G}
Let $(M_n)_{n\geq 1}$ be a sequence in $\mathbb S^m$ such that
\[
M_n\to M,
\]
and let $(G_n)_{n\geq 1}$ satisfy
\[
G_n\in \mathcal G(M_n),
\qquad
G_n\to G
\]
in $\mathbb S^m$. Then $G\in \mathcal G(M)$.
\end{lemma}

\begin{proof}
Since $G_n\in \mathcal G(M_n)\subset\Delta$ for every $n$ and $\Delta$ is closed in
$\mathbb S^m$, we
have $G\in\Delta$. Moreover,
\[
\tr(G_nM_n)=\lambda_{\min}(M_n),
\qquad
n\geq 1.
\]
Passing to the limit and using continuity of the trace pairing and of
$\lambda_{\min}$ on $\mathbb S^m$, we obtain
\[
\tr(GM)=\lambda_{\min}(M).
\]
Because
\[
\lambda_{\min}(M)=\min_{H\in\Delta}\tr(HM),
\]
it follows that $G\in \mathcal G(M)$.
\end{proof}

The following lower-semicontinuity statement is a convenient consequence of standard facts on normal and tangent cones, Moreau decomposition, and weak convergence, and is stated separately for later use.

\subsection{Lower semicontinuity of projected norms}
The following lemma combines standard facts on normal and tangent cones, Moreau decomposition, and weak convergence; see, for example, \citep[Example~6.24]{RockafellarWets1998}, \citep[Ch.~6]{bauschke2011convex}, and \citep[Theorem~2.1]{billingsley1999convergence}.

\begin{lemma}[Lower semicontinuity of projected norms]
\label{lem:projected-norm-lsc}
Define
\[
\psi(x,G):=
\left\|
\Pi_{T_{\Omega}(x)}\bigl(a_G(x)\bigr)
\right\|^2,
\qquad
(x,G)\in\Omega\times\Delta.
\]
Then the following hold.
\begin{enumerate}
\item[(i)] The map $\psi$ is bounded and lower semicontinuous on
$\Omega\times\Delta$.
\item[(ii)] If $\rho_n,\rho\in\mathcal P_2(\Omega)$ satisfy $\rho_n\to\rho$ in
$W_2$, and if $G_n,G\in\Delta$ satisfy $G_n\to G$ in $\mathbb S^m$, then
\[
\int_{\Omega}\psi(x,G)\,\rho(dx)
\le
\liminf_{n\to\infty}
\int_{\Omega}\psi(x,G_n)\,\rho_n(dx).
\]
\end{enumerate}
\end{lemma}

\begin{proof}
Since $\Omega$ is compact and convex, it is closed and convex. Hence, for each
$x\in\Omega$, the tangent cone $T_{\Omega}(x)$ is a closed convex cone, and the
normal cone satisfies
\[
N_{\Omega}(x)=T_{\Omega}(x)^\circ;
\]
see, for example, \citet[Example~6.24]{RockafellarWets1998}. Therefore Moreau's
decomposition for closed convex cones gives
\[
\psi(x,G)
=
\dist\!\bigl(a_G(x),N_{\Omega}(x)\bigr)^2,
\qquad
(x,G)\in\Omega\times\Delta;
\]
see, for example, \citet[Ch.~6]{bauschke2011convex}.

Because $f$ and $\nabla f$ are continuous on compact $\Omega$, and $\Delta$ is
compact, the map $(x,G)\mapsto a_G(x)$ is continuous and bounded on
$\Omega\times\Delta$. Hence $\psi$ is bounded.

To prove lower semicontinuity, let $(x_n,G_n)\to(x,G)$ in $\Omega\times\Delta$,
and pass to a subsequence, not relabelled, such that
\[
\psi(x_n,G_n)\to \liminf_{k\to\infty}\psi(x_k,G_k).
\]
For each $n$, choose $z_n\in N_{\Omega}(x_n)$ such that
\[
\|a_{G_n}(x_n)-z_n\|^2
\le
\psi(x_n,G_n)+\frac1n.
\]
Since both $\psi$ and $(x,G)\mapsto a_G(x)$ are bounded on
$\Omega\times\Delta$, the sequence $(z_n)$ is bounded. Passing to a further
subsequence if necessary, we may assume that $z_n\to z$ in $\R^d$. For every
$y\in\Omega$,
\[
\langle z_n,y-x_n\rangle\le 0,
\]
because $z_n\in N_{\Omega}(x_n)$. Letting $n\to\infty$ yields
\[
\langle z,y-x\rangle\le 0,
\qquad
\forall\,y\in\Omega,
\]
so $z\in N_{\Omega}(x)$. Using continuity of $(x,G)\mapsto a_G(x)$, we obtain
\[
\psi(x,G)
\le
\|a_G(x)-z\|^2
=
\lim_{n\to\infty}\|a_{G_n}(x_n)-z_n\|^2
\le
\lim_{n\to\infty}\left(\psi(x_n,G_n)+\frac1n\right)
=
\liminf_{n\to\infty}\psi(x_n,G_n).
\]
This proves (i).

For (ii), define probability measures on $\Omega\times\Delta$ by
\[
\nu_n:=\rho_n\otimes\delta_{G_n},
\qquad
\nu:=\rho\otimes\delta_G.
\]
Let $h$ be bounded and continuous on $\Omega\times\Delta$. Then
\[
\int_{\Omega\times\Delta} h(x,H)\,\nu_n(dxdH)
=
\int_{\Omega} h(x,G_n)\,\rho_n(dx).
\]
Since $G_n\to G$ in $\mathbb S^m$ and $\Omega\times\Delta$ is compact,
$h(\cdot,G_n)\to h(\cdot,G)$ uniformly on $\Omega$. Hence
\[
\int_{\Omega} h(x,G_n)\,\rho_n(dx)
-
\int_{\Omega} h(x,G)\,\rho_n(dx)\to 0.
\]
Since $\rho_n\to\rho$ in $W_2$ and $\Omega$ is compact, $\rho_n$ converges
weakly to $\rho$. Therefore
\[
\int_{\Omega} h(x,G)\,\rho_n(dx)\to \int_{\Omega} h(x,G)\,\rho(dx).
\]
Thus
\[
\int_{\Omega\times\Delta} h(x,H)\,\nu_n(dxdH)
\to
\int_{\Omega\times\Delta} h(x,H)\,\nu(dxdH),
\]
that is, $\nu_n\Rightarrow \nu$ weakly as probability measures on
$\Omega\times\Delta$. Applying the Portmanteau theorem to the bounded lower
semicontinuous function $\psi$ gives
\[
\int_{\Omega\times\Delta}\psi(x,H)\,\nu(dxdH)
\le
\liminf_{n\to\infty}
\int_{\Omega\times\Delta}\psi(x,H)\,\nu_n(dxdH);
\]
see, for example, \citet[Theorem~2.1]{billingsley1999convergence}. This is
exactly the asserted inequality.
\end{proof}

\section{Proofs for Section~\ref{sec:smooth}}

\subsection{Proof of Proposition~\ref{prop:particle_gradient}}
\begin{proof}
Let
\[
\psi=(\psi_1,\ldots,\psi_N)\in(\R^d)^N
\]
be arbitrary. Since $\mathcal F_N$ is differentiable at
$x=(x_1,\ldots,x_N)$, we have
\begin{equation}
\label{eq:particle-gradient-euclidean-expansion}
\mathcal F_N(x+\varepsilon\psi)-\mathcal F_N(x)
=
\varepsilon\sum_{i=1}^N
\left\langle \nabla_{x_i}\mathcal F_N(x),\psi_i\right\rangle
+
o(|\varepsilon|),
\qquad
\varepsilon\to0.
\end{equation}

By Lemma~\ref{lem:compactly_supported_gradient_interpolation}, there exists
$\varphi\in C_c^\infty(\R^d)$ such that
\[
\nabla\varphi(x_i)=\psi_i,
\qquad
i=1,\ldots,N.
\]
Set
\[
\rho_N^\varepsilon:=(\Id+\varepsilon \nabla\varphi)_\#\rho_N.
\]
Since $\rho_N=N^{-1}\sum_{i=1}^N\delta_{x_i}$ and
$\nabla\varphi(x_i)=\psi_i$, we have
\[
\rho_N^\varepsilon
=
\frac1N\sum_{i=1}^N \delta_{x_i+\varepsilon\psi_i}.
\]
Hence $\rho_N^\varepsilon$ is precisely the empirical measure associated with
the perturbed particle vector $x+\varepsilon\psi$. By definition of $\mathcal F_N$,
\begin{equation}
\label{eq:particle-gradient-empirical-identity}
\mathcal F_N(x+\varepsilon\psi)=\mathcal F(\rho_N^\varepsilon),
\qquad
\mathcal F_N(x)=\mathcal F(\rho_N).
\end{equation}

Moreover, $\nabla\varphi\in \operatorname{Tan}_{\rho_N}\mathcal P_2(\R^d)$, so
Lemma~\ref{lem:pushforward-characterization} gives
\begin{align}
\mathcal F(\rho_N^\varepsilon)-\mathcal F(\rho_N)
&=
\varepsilon
\int_{\R^d}
\left\langle
\nabla_{W_2}\mathcal F(\rho_N)(x),\nabla\varphi(x)
\right\rangle
\,\rho_N(dx)
+
o(|\varepsilon|) \notag\\
&=
\frac{\varepsilon}{N}
\sum_{i=1}^N
\left\langle
\nabla_{W_2}\mathcal F(\rho_N)(x_i),\nabla\varphi(x_i)
\right\rangle
+
o(|\varepsilon|) \notag\\
&=
\frac{\varepsilon}{N}
\sum_{i=1}^N
\left\langle
\nabla_{W_2}\mathcal F(\rho_N)(x_i),\psi_i
\right\rangle
+
o(|\varepsilon|).
\label{eq:particle-gradient-wasserstein-expansion}
\end{align}
Combining \eqref{eq:particle-gradient-empirical-identity} and
\eqref{eq:particle-gradient-wasserstein-expansion}, we obtain
\begin{equation}
\label{eq:particle-gradient-comparison}
\mathcal F_N(x+\varepsilon\psi)-\mathcal F_N(x)
=
\mathcal F(\rho_N^\varepsilon)-\mathcal F(\rho_N)
=
\frac{\varepsilon}{N}
\sum_{i=1}^N
\left\langle
\nabla_{W_2}\mathcal F(\rho_N)(x_i),\psi_i
\right\rangle
+
o(|\varepsilon|).
\end{equation}
Comparing \eqref{eq:particle-gradient-comparison} with
\eqref{eq:particle-gradient-euclidean-expansion}, dividing by
$\varepsilon$, and letting $\varepsilon\to0$, we find that
\[
\sum_{i=1}^N
\left\langle
\nabla_{x_i}\mathcal F_N(x)
-
\frac1N\nabla_{W_2}\mathcal F(\rho_N)(x_i),
\psi_i
\right\rangle
=0
\]
for every $\psi\in(\R^d)^N$. Since the vectors $\psi_i$ are arbitrary, this implies
\[
\nabla_{x_i}\mathcal F_N(x)
=
\frac1N\nabla_{W_2}\mathcal F(\rho_N)(x_i),
\qquad
i=1,\ldots,N.
\]
This proves \eqref{eq:particle-gradient}.
\end{proof}

\begin{remark}[On coincident particles]
The pairwise-distinctness assumption is used only to prescribe arbitrary particle
velocities by a single smooth gradient field. The same conclusion extends to
coincident particles by first grouping identical locations, which yields the
identity after summing over each coincidence class, and then using the permutation
symmetry of $\mathcal F_N$ to recover the componentwise formula.
\end{remark}

\subsection{Proof of Proposition~\ref{prop:particle_value}}

We first record an auxiliary empirical-approximation lemma.

\begin{lemma}[Empirical $W_2$ approximation by finitely supported measures]
\label{lem:empirical-w2-approx}
Assume that $\Omega$ is compact. Let $\mu$ be a probability measure on $\Omega$, and let
\[
\mu_N:=\frac1N\sum_{i=1}^N \delta_{X_i},
\]
where $X_1,\ldots,X_N$ are independent random variables with common law $\mu$.
Then there exists $C_0>0$ such that, for all sufficiently large $N$, there exists a deterministic measure
\[
\widehat\mu_N\in\mathcal P_N(\Omega)
\]
satisfying
\[
W_2(\widehat\mu_N,\mu)\le C_0 r_N,
\]
where
\[
r_N:=
\begin{cases}
N^{-1/4}, & d<4,\\[0.3ex]
N^{-1/4}(\log N)^{1/2}, & d=4,\\[0.3ex]
N^{-1/d}, & d>4.
\end{cases}
\]
\end{lemma}

\begin{proof}
Since $\Omega$ is compact, $\mu$ is compactly supported. By Theorem~2 of \citet{fournier2015rate}, specialized to $p=2$, there exist constants $C,c>0$, depending only on $\mu$ and $d$, such that for all $\varepsilon\in(0,1]$,
\[
\mathbb P\!\left(W_2^2(\mu_N,\mu)\ge \varepsilon\right)
\le
\begin{cases}
C\exp(-cN\varepsilon^2), & d<4,\\[0.6ex]
C\exp\!\left(
-cN\left(\dfrac{\varepsilon}{\log(2+1/\varepsilon)}\right)^2
\right), & d=4,\\[1.2ex]
C\exp(-cN\varepsilon^{d/2}), & d>4.
\end{cases}
\]
Set
\[
\varepsilon=C_0^2 r_N^2.
\]
Since $r_N\to0$, we have $\varepsilon\in(0,1]$ for all sufficiently large $N$. Then, for all sufficiently large $N$,
\[
\mathbb P\!\left(W_2^2(\mu_N,\mu)\ge C_0^2 r_N^2\right)
\le
\begin{cases}
C\exp(-cC_0^4), & d<4,\\[0.6ex]
C\exp(-cC_0^4), & d=4,\\[0.6ex]
C\exp(-cC_0^d), & d>4,
\end{cases}
\]
where in the case $d=4$ we used that, as $N\to\infty$,
\[
\log\!\left(2+\frac{1}{C_0^2N^{-1/2}\log N}\right)\asymp \log N.
\]
Choosing $C_0$ sufficiently large, we obtain
\[
\mathbb P\!\left(W_2^2(\mu_N,\mu)\ge C_0^2 r_N^2\right)<1
\]
for all sufficiently large $N$. Therefore
\[
\mathbb P\!\left(W_2(\mu_N,\mu)\le C_0 r_N\right)>0
\]
for all sufficiently large $N$. Since every realization of $\mu_N$ belongs to $\mathcal P_N(\Omega)$, there exists a deterministic measure $\widehat\mu_N\in\mathcal P_N(\Omega)$ such that
\[
W_2(\widehat\mu_N,\mu)\le C_0 r_N.
\]
\end{proof}

\begin{proof}[of Proposition~\ref{prop:particle_value}]
Apply Lemma~\ref{lem:empirical-w2-approx} with $\mu=\rho^\ast$. Then there exists $C_0>0$ such that, for all sufficiently large $N$, one can find
\[
\widehat\rho_N\in\mathcal P_N(\Omega)
\]
satisfying
\[
W_2(\widehat\rho_N,\rho^\ast)\le C_0 r_N.
\]
Because $r_N\to0$, we have $C_0r_N\le\delta$ for all sufficiently large $N$. By the assumed local $W_2$-Lipschitz continuity of $\mathcal F$ at $\rho^\ast$,
\[
\mathcal F(\widehat\rho_N)-\mathcal F(\rho^\ast)
\le
\bigl|\mathcal F(\widehat\rho_N)-\mathcal F(\rho^\ast)\bigr|
\le
L\,W_2(\widehat\rho_N,\rho^\ast)
\le
LC_0r_N.
\]
Since $\rho^\ast$ minimizes $\mathcal F$ over $\mathcal P_2(\Omega)$,
\[
\inf_{\rho\in\mathcal P_2(\Omega)}\mathcal F(\rho)
=
\mathcal F(\rho^\ast)
\le
\inf_{\rho_N\in\mathcal P_N(\Omega)}\mathcal F(\rho_N)
\le
\mathcal F(\widehat\rho_N).
\]
Hence
\[
0
\le
\inf_{\rho_N\in\mathcal P_N(\Omega)}\mathcal F(\rho_N)
-
\inf_{\rho\in\mathcal P_2(\Omega)}\mathcal F(\rho)
\le
\mathcal F(\widehat\rho_N)-\mathcal F(\rho^\ast)
\le
LC_0r_N,
\]
which proves \eqref{eq:particle-value-rate}.
\end{proof}

\section{Proofs for Section~\ref{sec:directional-derivative}}
\subsection{Proof of Proposition~\ref{prop:directional-derivative}}
\begin{proof}
For $1\le i,j\le m$, define $\phi_{ij}(x):=f_i(x)f_j(x)$.
Since $f\in C^1(\Omega;\R^m)$ and $\Omega$ is compact, each $\phi_{ij}$ is
Lipschitz on $\Omega$. Because $(\rho_t)_{|t|<\varepsilon}$ is absolutely
continuous in $\mathcal P_2(\Omega)$, the map
$t\mapsto \int_\Omega \phi_{ij}\,d\rho_t$ is absolutely continuous on
$(-\varepsilon,\varepsilon)$.

\textit{Step 1: Differentiation of the information matrix.}
Testing the continuity equation
$\partial_t\rho_t+\nabla\!\cdot(\rho_tu_t)=0$ against $\eta(t)\phi_{ij}(x)$
with $\eta\in C_c^\infty(-\varepsilon,\varepsilon)$ gives
\[
\frac{d}{dt}\int_\Omega \phi_{ij}\,d\rho_t
=
\int_\Omega \nabla\phi_{ij}(x)\cdot u_t(x)\,\rho_t(dx)
\]
for almost every $t$; see \citet[Ch.~8]{ambrosio_gradient_2008}.
Since $(M_{\rho_t})_{ij}=\int_\Omega \phi_{ij}\,d\rho_t$, the map
$t\mapsto M_{\rho_t}$ is absolutely continuous in $\mathbb S^m$ and
\begin{equation}
\label{eq:dMdt}
\frac{d}{dt}M_{\rho_t}
=
\int_\Omega
\bigl\{
\nabla f(x)\,u_t(x)\,f(x)^\T
+
f(x)\,\bigl(\nabla f(x)\,u_t(x)\bigr)^\T
\bigr\}\,\rho_t(dx)
\end{equation}
for almost every $t$.

\textit{Step 2: Composition with $\lambda_{\min}$.}
The map $\lambda_{\min}:\mathbb S^m\to\R$ is Lipschitz, so
$t\mapsto \mathcal F_E(\rho_t)=\lambda_{\min}(M_{\rho_t})$ is absolutely
continuous. At any $t$ where both $t\mapsto M_{\rho_t}$ and
$t\mapsto \mathcal F_E(\rho_t)$ are differentiable, the variational
representation $\lambda_{\min}(M)=\min_{G\in\Delta}\tr(GM)$ and Danskin's
theorem \citep[Prop.~B.25]{Bertsekas1999} yield
\[
\frac{d}{dt}\mathcal F_E(\rho_t)
=
\min_{G\in \mathcal{G}(M_{\rho_t})}
\tr\!\Bigl(G\,\frac{d}{dt}M_{\rho_t}\Bigr).
\]

\textit{Step 3: Trace identification.}
For any $G\in \mathcal{G}(M_{\rho_t})$, substituting \eqref{eq:dMdt} and applying
cyclicity of the trace gives
\[
\tr\!\Bigl(G\,\frac{d}{dt}M_{\rho_t}\Bigr)
=
\int_\Omega
\langle 2\nabla f(x)^\T Gf(x),\,u_t(x)\rangle
\,\rho_t(dx)
=
\langle a_G,u_t\rangle_{\rho_t}.
\]
Hence
\[
\frac{d}{dt}\mathcal F_E(\rho_t)
=
\min_{G\in \mathcal{G}(M_{\rho_t})}\langle a_G,u_t\rangle_{\rho_t}
=
D\mathcal F_E(\rho_t)[u_t]
\]
for almost every $t\in(-\varepsilon,\varepsilon)$.
\end{proof}
\section{Proofs for Section~\ref{sec:steepest-ascent}}

\subsection{Proof of Proposition~\ref{prop:gap-representation}}
\begin{proof}
The map $G\mapsto\|\Pi_{T_{\Omega}(\cdot)}(a_G(\cdot))\|_{\rho}$ is
continuous on $\mathcal G(M_{\rho})$: this follows from the linearity of
$G\mapsto a_G$, the boundedness of $f$ and $\nabla f$ on $\Omega$, and the
$1$-Lipschitz property of orthogonal projection onto a closed convex cone.
Since $\mathcal G(M_{\rho})$ is compact by Proposition~\ref{prop:G-structure},
the Weierstrass theorem gives the existence of a
minimizer~$G_{\rho}^{\star}$ in~\eqref{eq:Gstar-def}.

It remains to verify~\eqref{eq:steepest-field-explicit}.
By Proposition~\ref{prop:G-structure}, $\mathcal G(M_{\rho})$ is nonempty,
compact, and convex. The set $K_{\rho}$ is convex, closed, and bounded in
$L^2(\rho;\R^d)$; since this space is reflexive, $K_{\rho}$ is weakly compact.
For fixed $G$, the map $u\mapsto \langle a_G,u\rangle_{\rho}$ is linear and
weakly continuous; for fixed $u$, the map
$G\mapsto \langle a_G,u\rangle_{\rho}$ is continuous and affine. Sion's
minimax theorem \citep{Sion1958} therefore gives
\begin{equation}
\label{eq:sion-gap}
\sup_{u\in K_{\rho}}\min_{G\in\mathcal G(M_{\rho})}
\langle a_G,u\rangle_{\rho}
=
\min_{G\in\mathcal G(M_{\rho})}\sup_{u\in K_{\rho}}
\langle a_G,u\rangle_{\rho}.
\end{equation}

Fix $G\in\mathcal G(M_{\rho})$. Since $T_{\Omega}(x)$ is a closed convex
cone, Moreau's decomposition \citep[Ch.~6]{bauschke2011convex} gives
\[
a_G(x)
=
\Pi_{T_{\Omega}(x)}\bigl(a_G(x)\bigr)
+
n_G(x),
\qquad
n_G(x)\in T_{\Omega}(x)^\circ,
\]
with
$\bigl\langle\Pi_{T_{\Omega}(x)}(a_G(x)),\,n_G(x)\bigr\rangle=0$
for $\rho$-almost every $x$.
If $u\in K_{\rho}$, then $u(x)\in T_{\Omega}(x)$ and $\|u\|_{\rho}\le1$, so
$\langle n_G(x),u(x)\rangle\le0$ pointwise $\rho$-almost everywhere. Hence
\[
\langle a_G,u\rangle_{\rho}
\le
\bigl\langle\Pi_{T_{\Omega}(\cdot)}(a_G(\cdot)),\,u\bigr\rangle_{\rho}
\le
\bigl\|\Pi_{T_{\Omega}(\cdot)}(a_G(\cdot))\bigr\|_{\rho}\,\|u\|_{\rho}
\le
\bigl\|\Pi_{T_{\Omega}(\cdot)}(a_G(\cdot))\bigr\|_{\rho}.
\]
If $\|\Pi_{T_{\Omega}(\cdot)}(a_G(\cdot))\|_{\rho}>0$, equality is attained
by $u_G:=\Pi_{T_{\Omega}(\cdot)}(a_G(\cdot))/
\|\Pi_{T_{\Omega}(\cdot)}(a_G(\cdot))\|_{\rho}\in K_{\rho}$, by Moreau
orthogonality; if $\|\Pi_{T_{\Omega}(\cdot)}(a_G(\cdot))\|_{\rho}=0$, both
sides vanish. Therefore
\begin{equation}
\label{eq:sup-u-identity}
\sup_{u\in K_{\rho}}\langle a_G,u\rangle_{\rho}
=
\bigl\|\Pi_{T_{\Omega}(\cdot)}\bigl(a_G(\cdot)\bigr)\bigr\|_{\rho}.
\end{equation}

Substituting~\eqref{eq:sup-u-identity} into the right-hand side
of~\eqref{eq:sion-gap}, the minimum over
$\mathcal G(M_{\rho})$ is attained at~$G_{\rho}^{\star}$, so
\begin{equation}
\label{eq:supmin-value}
\sup_{u\in K_{\rho}}D\mathcal F_E(\rho)[u]
=
\bigl\|\Pi_{T_{\Omega}(\cdot)}\bigl(a_{G_{\rho}^{\star}}(\cdot)\bigr)
\bigr\|_{\rho}.
\end{equation}

If $\|\Pi_{T_{\Omega}(\cdot)}(a_{G_{\rho}^{\star}}(\cdot))\|_{\rho}=0$,
then~\eqref{eq:supmin-value} gives
$\sup_{u\in K_{\rho}}D\mathcal F_E(\rho)[u]=0$, so
$\bar{\nabla}_{W_2}\mathcal F_E(\rho)=0
=\Pi_{T_{\Omega}(\cdot)}(a_{G_{\rho}^{\star}}(\cdot))$
by Definition~\ref{def:steepest-ascent}.

Suppose
$\|\Pi_{T_{\Omega}(\cdot)}(a_{G_{\rho}^{\star}}(\cdot))\|_{\rho}>0$.
The map $u\mapsto\min_{G\in\mathcal G(M_{\rho})}
\langle a_G,u\rangle_{\rho}$ is the infimum of weakly continuous affine
functions, hence weakly upper semicontinuous. Since $K_{\rho}$ is weakly
compact, this map attains its supremum at some $\hat u\in K_{\rho}$. Then
\[
\bigl\|\Pi_{T_{\Omega}(\cdot)}\bigl(a_{G_{\rho}^{\star}}(\cdot)\bigr)
\bigr\|_{\rho}
=
\min_{G\in\mathcal G(M_{\rho})}
\langle a_G,\hat u\rangle_{\rho}
\le
\langle a_{G_{\rho}^{\star}},\hat u\rangle_{\rho}
\le
\bigl\|\Pi_{T_{\Omega}(\cdot)}\bigl(a_{G_{\rho}^{\star}}(\cdot)\bigr)
\bigr\|_{\rho},
\]
where the last inequality is~\eqref{eq:sup-u-identity} applied with
$G=G_{\rho}^{\star}$. All inequalities are therefore equalities, so $\hat u$
attains the supremum in~\eqref{eq:sup-u-identity} for
$G=G_{\rho}^{\star}$. By the Cauchy--Schwarz equality condition in the chain
leading to~\eqref{eq:sup-u-identity}, this maximizer is unique and equals
\[
u_{\rho}^{\star}
=
\frac{\Pi_{T_{\Omega}(\cdot)}(a_{G_{\rho}^{\star}}(\cdot))}
{\|\Pi_{T_{\Omega}(\cdot)}(a_{G_{\rho}^{\star}}(\cdot))\|_{\rho}}.
\]
By Definition~\ref{def:steepest-ascent}
and~\eqref{eq:supmin-value},
\[
\bar{\nabla}_{W_2}\mathcal F_E(\rho)
=
\bigl\|\Pi_{T_{\Omega}(\cdot)}\bigl(a_{G_{\rho}^{\star}}(\cdot)\bigr)
\bigr\|_{\rho}
\cdot u_{\rho}^{\star}
=
\Pi_{T_{\Omega}(\cdot)}\bigl(a_{G_{\rho}^{\star}}(\cdot)\bigr),
\]
which is~\eqref{eq:steepest-field-explicit}.
\end{proof}

\subsection{Proof of Proposition~\ref{prop:gap-stationary}}
\begin{proof}
By Definition~\ref{def:steepest-ascent},
$\bar{\nabla}_{W_2}\mathcal F_E(\rho)=m_{\Omega}(\rho)\,u_{\rho}^{\star}$,
so $\bar{\nabla}_{W_2}\mathcal F_E(\rho)=0$ if and only if
$m_{\Omega}(\rho)=0$. It therefore suffices to show that $\rho$ is stationary
if and only if $m_{\Omega}(\rho)=0$.

Since $0\in K_{\rho}$, we have $m_{\Omega}(\rho)\ge D\mathcal F_E(\rho)[0]=0$.
If $\rho$ is stationary, then $D\mathcal F_E(\rho)[u]\le0$ for every
$u\in K_{\rho}$, so $m_{\Omega}(\rho)\le0$ and hence $m_{\Omega}(\rho)=0$.

Conversely, suppose $m_{\Omega}(\rho)=0$, and let $w\in L^2(\rho;\R^d)$
satisfy $w(x)\in T_{\Omega}(x)$ for $\rho$-almost every $x$.
If $w=0$, then $D\mathcal F_E(\rho)[w]=0\le0$.
If $w\neq0$, set $u:=w/\|w\|_{\rho}\in K_{\rho}$; positive homogeneity
of~\eqref{eq:directional-derivative} in the direction variable gives
\[
D\mathcal F_E(\rho)[w]
=
\|w\|_{\rho}\,D\mathcal F_E(\rho)[u]
\le
\|w\|_{\rho}\,m_{\Omega}(\rho)
=0.
\]
Hence $\rho$ is stationary in the sense of Definition~\ref{def:stationary}.
\end{proof}

\section{Proofs for Section~\ref{sec:flow}}

\subsection{Proof of Proposition~\ref{prop:gap-lsc}}
\begin{proof}
Let $\rho_n\to\rho$ in $\mathcal P_2(\Omega)$. For each $n$, choose
\[
G_n\in\mathcal G(M_{\rho_n})
\]
such that
\[
m_{\Omega}(\rho_n)
=
\left\|
\Pi_{T_{\Omega}(\cdot)}\bigl(a_{G_n}(\cdot)\bigr)
\right\|_{\rho_n}.
\]
Passing to a subsequence if necessary, we may assume that
\[
\liminf_{n\to\infty}m_{\Omega}(\rho_n)^2
=
\lim_{n\to\infty}m_{\Omega}(\rho_n)^2.
\]
Since $\Delta$ is compact, after passing to a further subsequence we may assume
that
\[
G_n\to G\in\Delta.
\]

Because $\Omega$ is compact and $x\mapsto f(x)f(x)^\T$ is continuous on
$\Omega$, the weak convergence $\rho_n\Rightarrow \rho$ implies
\[
M_{\rho_n}
=
\int_{\Omega} f(x)f(x)^\T\,\rho_n(dx)
\to
\int_{\Omega} f(x)f(x)^\T\,\rho(dx)
=
M_{\rho}.
\]
Lemma~\ref{lem:closed-graph-G} therefore yields $G\in \mathcal G(M_{\rho})$.

Applying Lemma~\ref{lem:projected-norm-lsc} with the same sequence
$(\rho_n,G_n)$ gives
\[
\left\|
\Pi_{T_{\Omega}(\cdot)}\bigl(a_G(\cdot)\bigr)
\right\|_{\rho}^2
\le
\liminf_{n\to\infty}
\left\|
\Pi_{T_{\Omega}(\cdot)}\bigl(a_{G_n}(\cdot)\bigr)
\right\|_{\rho_n}^2.
\]
Since $G\in \mathcal G(M_{\rho})$, Proposition~\ref{prop:gap-representation} implies
\[
m_{\Omega}(\rho)^2
\le
\left\|
\Pi_{T_{\Omega}(\cdot)}\bigl(a_G(\cdot)\bigr)
\right\|_{\rho}^2
\le
\liminf_{n\to\infty}m_{\Omega}(\rho_n)^2.
\]
Because all terms are nonnegative, we conclude that
\[
m_{\Omega}(\rho)\le \liminf_{n\to\infty}m_{\Omega}(\rho_n).
\]
\end{proof}

\section{Proofs for Section~\ref{sec:particle-computation}}

\subsection{Proof of Proposition~\ref{prop:particle-convex}}
\begin{proof}
Write $\rho_N=N^{-1}\sum_{i=1}^N\delta_{x_i}$ and $M_N=M_{\rho_N}$, and let
$V_N\in\R^{m\times s_N}$ have orthonormal columns spanning the eigenspace of
$M_N$ associated with $\lambda_{\min}(M_N)$.

By Proposition~\ref{prop:G-structure}, every $G\in \mathcal G(M_N)$ takes the form
$G=V_NSV_N^\T$ for some $S\succeq 0$ with $\tr(S)=1$.

For such $S$, set
\[
a_i(S):=2\nabla f(x_i)^\T V_NSV_N^\T f(x_i),
\qquad
v_i(S):=\Pi_{T_{\Omega}(x_i)}\bigl(a_i(S)\bigr),
\qquad
i=1,\ldots,N.
\]
Since $a_G(x_i)=a_i(S)$ and
$\Pi_{T_\Omega(x_i)}(a_G(x_i))=v_i(S)$ whenever $G=V_NSV_N^\T$, substituting
into Proposition~\ref{prop:gap-representation} gives
\[
m_{\Omega}(\rho_N)
=
\min_{S\succeq 0,\,\tr(S)=1}
\left(
\frac1N\sum_{i=1}^N \|v_i(S)\|^2
\right)^{\!1/2},
\]
which coincides with~\eqref{eq:particle-direction-sdp}.

The feasible set $\{S\in\mathbb S^{s_N}:S\succeq 0,\,\tr(S)=1\}$ is a
compact convex subset of~$\mathbb S^{s_N}$.

To establish convexity of the objective, equip
$\mathcal H_N:=(\R^d)^N$ with the inner product
$\langle u,w\rangle_N:=N^{-1}\sum_{i=1}^N\langle u_i,w_i\rangle$ and induced
norm $\|\cdot\|_N$, and write
\[
\mathcal C_N:=T_{\Omega}(x_1)\times\cdots\times T_{\Omega}(x_N),
\qquad
\mathcal K_N:=\{u\in\mathcal C_N:\|u\|_N\le 1\}.
\]
Set $a(S):=(a_1(S),\ldots,a_N(S))$ and $v(S):=(v_1(S),\ldots,v_N(S))$.
Because projection onto a product of cones acts componentwise,
$v(S)=\Pi_{\mathcal C_N}(a(S))$.

The same Moreau-decomposition argument used in the proof of
Proposition~\ref{prop:gap-representation} yields
\[
\|v(S)\|_N
=
\sup_{u\in\mathcal K_N}\langle a(S),u\rangle_N.
\]
For each fixed $u\in\mathcal K_N$, the map
$S\mapsto\langle a(S),u\rangle_N$ is affine in~$S$, since $a(S)$ depends
affinely on~$S$.  Hence $S\mapsto\|v(S)\|_N$ is convex as the pointwise
supremum of affine functions. Continuity follows from the affinity of
$S\mapsto a(S)$ and the continuity of projection onto a closed convex cone.

Therefore~\eqref{eq:particle-direction-sdp} is a finite-dimensional convex
problem and admits a minimizer by the Weierstrass theorem.

Let $S_N^\star$ be any minimizer of~\eqref{eq:particle-direction-sdp} and set
$G_N^\star:=V_NS_N^\star V_N^\T\in \mathcal G(M_N)$.  Then, for $i=1,\ldots,N$,
\[
v_i(S_N^\star)
=
\Pi_{T_{\Omega}(x_i)}\bigl(a_{G_N^\star}(x_i)\bigr),
\]
so $(v_1(S_N^\star),\ldots,v_N(S_N^\star))$ is the particle realization of the
constrained Wasserstein steepest-ascent field at~$\rho_N$.  The corresponding
particle stationarity measure satisfies
\[
\widehat m_N(x_1,\ldots,x_N)
:=
\left(
\frac1N\sum_{i=1}^N \|v_i(S_N^\star)\|^2
\right)^{\!1/2}
=
m_{\Omega}(\rho_N).
\]
When $\widehat m_N>0$, normalizing by $\widehat m_N$ yields a maximizer
of~\eqref{eq:particle-direction-problem}; when $\widehat m_N=0$, the zero
vector is feasible and stationary.
\end{proof}

\section{Particle swarm optimization implementation}
\label{sec:pso-supp}

For the numerical benchmarks of Section~\ref{sec:experiments}, we use
the particle swarm optimization implementation of
\citet{biswas2014pso}. Each swarm member encodes an equal-weight
$N$-point design, with the same support size $N=100$ as the particle
methods of Algorithm~\ref{alg:particle-ascent} and
Corollary~\ref{cor:particle_lift}, and is represented by the
concatenation $y\in\R^{Nd}$ of its $N$ support points. At iteration
$t$, swarm member $i$ has position $y_i^{(t)}$ and velocity
$V_i^{(t)}$, and maintains its own best past position $p_i^{(t)}$; the
swarm's global best is $g^{(t)}$. The update is
\begin{align*}
V_i^{(t+1)}
&= w_t V_i^{(t)}
 + c_1\, r_{1,i}^{(t)} \circ \bigl(p_i^{(t)}-y_i^{(t)}\bigr)
 + c_2\, r_{2,i}^{(t)} \circ \bigl(g^{(t)}-y_i^{(t)}\bigr),
\\
y_i^{(t+1)}
&= \Pi_{B}\bigl(y_i^{(t)} + V_i^{(t+1)}\bigr),
\end{align*}
where $r_{1,i}^{(t)}$ and $r_{2,i}^{(t)}$ are independent vectors with
independent uniform entries on $[0,1]$, the symbol $\circ$ denotes
componentwise multiplication, and $\Pi_B$ clips each of the $Nd$
coordinates of its argument to the smallest axis-aligned hypercube $B$
containing $\Omega$. For cube design regions, $B=\Omega$; for the unit
ball, $B=[-1,1]^k$, and before the information matrix of a swarm
member is formed, any support point $x$ with $\|x\|>1$ is radially
rescaled to $x/\|x\|$, so that every support point used in the
objective lies in $\Omega$. The cognitive and social coefficients are
$c_1=c_2=2$, and the inertia weight decreases linearly from
$w_{\max}=0.9$ at $t=1$ to $w_{\min}=0.4$ at $t=T$. Each component of
$V_i^{(t+1)}$ is further clamped to $[-V_{\max},V_{\max}]$, with
$V_{\max}$ equal to $0.2$ times the coordinate upper bound of $B$.
Initial positions are drawn independently and uniformly on $B$, and
initial velocities are drawn componentwise uniformly on
$[-V_{\max},V_{\max}]$. The swarm size is $100$, the iteration budget
is $T=1000$, and each experiment is repeated over $100$ independent
runs; the best, mean and worst objective values are reported in
Table~\ref{tab:results}.

\bibliographystyle{plainnat}
\bibliography{references}

\end{document}